\documentclass[draft,oneside]{llncs}
\usepackage{amsmath,amssymb,amsfonts}
\usepackage{multirow}
\usepackage{algpseudocode}
\usepackage{algorithm2e}
\title{
    Fast, uniform, and compact scalar multiplication 
    \\
    for elliptic curves and genus~2 Jacobians 
    \\
    with applications to signature schemes
}

\newcommand{\ZZ}{\ensuremath{\mathbb{Z}}}
\newcommand{\FF}{\ensuremath{\mathbb{F}}}
\newcommand{\PP}{\ensuremath{\mathbb{P}}}
\newcommand{\EC}{\ensuremath{\mathcal{E}}}
\newcommand{\C}{\ensuremath{\mathcal{C}}}
\newcommand{\Jac}[1]{\ensuremath{\mathcal{J}_{#1}}}
\newcommand{\Kfast}[1]{\ensuremath{\mathcal{K}_{#1}^{\mathrm{fast}}}}
\newcommand{\Kgen}[1]{\ensuremath{\mathcal{K}_{#1}^{\mathrm{gen}}}}
\newcommand{\subgrp}[1]{\ensuremath{\left\langle{#1}\right\rangle}}

\newcommand{\G}{\ensuremath{\mathcal{G}}}
\newcommand{\Gquotient}{\ensuremath{\mathcal{G}/\subgrp{\pm1}}}

\newcommand{\ADD}{\ensuremath{\mathtt{ADD}}}
\newcommand{\xADD}{\ensuremath{\mathtt{xADD}}}
\newcommand{\xDBL}{\ensuremath{\mathtt{xDBL}}}
\newcommand{\xDBLADD}{\ensuremath{\mathtt{xDBLADD}}}

\newcommand{\Project}{\ensuremath{\mathtt{Project}}}
\newcommand{\Recover}{\ensuremath{\mathtt{Recover}}}

\newcommand{\Pattern}{{\ensuremath{\Project}-pseudomultiply-\ensuremath{\Recover}}}

\newcommand{\FqM}[1]{{{#1}\mathbf{M}}}
\newcommand{\FqE}[1]{{{#1}\mathbf{E}}}
\newcommand{\FqS}[1]{{{#1}\mathbf{S}}}
\newcommand{\FqI}[1]{{{#1}\mathbf{I}}}
\newcommand{\Fqa}[1]{{{#1}\mathbf{a}}}
\newcommand{\Fqc}[1]{{{#1}\mathbf{m}_c}}

\author{Ping Ngai Chung\inst{1} \and Craig Costello\inst{2} \and Benjamin Smith\inst{3}}
\institute{ 
    University of Chicago, USA\\
    \email{briancpn@math.uchicago.edu}\\
    \and
    Microsoft Research, USA \\
    \email{craigco@microsoft.com} \\
    \and
    INRIA 
    \emph{and}
    Laboratoire d'Informatique de l'\'{E}cole polytechnique (LIX), 
    France \\
    \email{smith@lix.polytechnique.fr}  
}

\begin{document}
\maketitle

\begin{abstract}
    We give a general framework for uniform, constant-time 
    one- and two-dimensional scalar multiplication algorithms for 
    elliptic curves and Jacobians of genus~2 curves that operate by
    projecting to the \(x\)-line or Kummer surface, where we can 
    exploit faster and more uniform pseudomultiplication, before 
    recovering the proper ``signed'' output back on the curve or Jacobian.  
    This extends the work of L\'opez and Dahab, Okeya and Sakurai, 
    and Brier and Joye to genus~2, and also to two-dimensional scalar
    multiplication.
    Our results show that many existing fast pseudomultiplication
    implementations (hitherto limited to applications in 
    Diffie--Hellman key exchange) can be wrapped with simple and 
    efficient pre- and post-computations to yield competitive full 
    scalar multiplication algorithms, ready for use in more general
    discrete logarithm-based cryptosystems, including signature schemes.
    This is especially interesting for genus~2, where Kummer surfaces
    can outperform comparable elliptic curve systems.
    As an example, we construct an instance of the Schnorr signature scheme
    driven by Kummer surface arithmetic.
\end{abstract}

\section{
    Introduction
}

In terms of per-bit security, 
elliptic curves and Jacobians of genus~2 curves appear to be roughly equivalent.
However, when it comes to efficient and side-channel-aware implementations, 
we see a curious divergence.
Full genus~2 Jacobian arithmetic is relatively slow compared with
elliptic curves,
and hard to implement in a uniform and constant-time fashion.
On the other hand, we have an extremely fast and uniform scalar
pseudomultiplication for genus~2 Kummer surfaces, 
which often outperforms its elliptic equivalent;
but this comes at the cost of identifying elements with their inverses,
so Kummer surfaces are widely believed to be suitable only for 
Diffie--Hellman protocols, where no individual full group operations are required.

Thus, genus~2 beats elliptic curve performance for Diffie--Hellman
(where we can use Kummer surfaces), but is beaten in almost every
other implementation scenario (where we are confined to Jacobians).
In this article we show how to exploit Kummer surface arithmetic to
carry out full Jacobian scalar multiplications, bringing the speed
and side-channel security of Kummers to implementations of other 
discrete-log-based cryptographic protocols.  
In particular, our results show that many existing competitive 
Diffie--Hellman implementations based on pseudomultiplication
can be wrapped with simple and efficient pre- and post-computations 
to yield competitive full scalar multiplication algorithms, 
ready for use in more general cryptosystems, 
including signature schemes.

\paragraph{Scalar multiplication.}
To make things precise,
let \(\G\) be a subgroup of an elliptic curve or a genus~2 Jacobian
(with \(\oplus\) denoting the group law, and \(\ominus\) its inverse).
We are interested in computing one-dimensional scalar multiplications
\[
    (m,P) 
    \longmapsto 
    [m]P 
    := 
    \underbrace{P\oplus\cdots\oplus P}_{m \text{ times }}
    \quad
    \text{for}
    \quad
    m \in \ZZ_{\ge 0}, P \in \G
    \ ,
\]
which is the operation at the heart of all discrete
logarithm and Diffie--Hellman problem-based cryptosystems.
We are also interested in two-dimensional multiscalar multiplications
\[
    ((m,n),(P,Q))
    \longmapsto
    [m]P\oplus[n]Q
    \quad
    \text{for}
    \quad
    m,n \in \ZZ_{\ge 0}, P,Q \in \G
    \ ;
\]
these appear explicitly in many cryptographic protocols, 
including signature verification, but they are also a key ingredient 
in endomorphism-accelerated one-dimensional scalar multiplication techniques
such as GLV~\cite{GLV} and its descendants.

If the scalar \(m\) is secret, then 
\([m]P\) must be computed in a \emph{uniform} and \emph{constant-time} way 
to protect against even the most elementary side-channel attacks.
This means that the execution path of the algorithm 
must be independent of the scalar \(m\)
(we may assume that the bitlength of \(m\) is fixed).

Uniform, constant-time algorithms for elliptic curve scalar
multiplication are well-known (and even widely-used).
In contrast, if \(\G\) is a subgroup of a genus~2 Jacobian, then 
this requirement has represented an insurmountable obstacle until now:
the usual Cantor arithmetic~\cite{cantor} and its derivatives~\cite{jac-on-jac}
have so many incompatible special cases 
that it has appeared impossible to implement it in a uniform way
without abandoning all hope of competitive efficiency
(see~\S\ref{sec:genus2}).

\paragraph{Kummer surfaces and \(x\)-lines.}
The situation changes dramatically when we pass to the quotient 
\(\Gquotient\), identifying elements with their inverses.
Let
\[
    x : \G \longrightarrow \Gquotient
\]
be the quotient map,
so \(x(P) = x(\ominus P)\) for all \(P\) in \(\G\).
In the elliptic case, 
\(x\) is projection onto the \(x\)-coordinate (whence our notation);
in genus~2, \(x\) is the map from the Jacobian to its Kummer surface.
We emphasize that \(\Gquotient\) is \emph{not} a group,
and at first glance this prevents us instantiating group-based protocols
in \(\Gquotient\).
Nevertheless, scalar multiplication in \(\G\)
induces a well-defined \emph{pseudomultiplication}
\[
    (m,x(P)) \longmapsto x([m]P)
\]
in \(\Gquotient\), because \([m](\pm P) = \pm([m]P)\).
This suffices for implementing protocols like Diffie--Hellman key
exchange which \emph{only} involve
scalar multiplication, and not individual group operations 
(which we lose in passing from \(\G\) to \(\Gquotient\)).

Pseudomultiplication algorithms are typically faster and simpler than
full scalar multiplication algorithms---the \(x\)-only Montgomery ladder
being a case in point for elliptic curves---and these algorithms have
therefore become the basis of the fastest and safest Diffie--Hellman
implementations (such as Bernstein's Curve25519
software~\cite{Curve25519} and its heirs). 
We also see excellent Diffie--Hellman implementations based on Kummer
surfaces~\cite{danjabe-and-nok}. 

But it is widely believed that pseudomultiplication cannot be used for
general cryptosystems, because \(\Gquotient\) is not a group: 
since the ``sign'' of the output of a pseudomultiplication is ambiguous,
it cannot be used as the input to individual group operations.
This belief has so far disqualified Kummers as candidates for implementing most
common signature schemes, as well as encryption schemes such as
ElGamal~\cite{ElGamal}.
As a result, we tend not to see highly competitive genus~2
implementations of signatures and public-key encryption schemes,
because we are hamstrung by relatively slow and non-uniform Jacobian arithmetic.

\paragraph{Recovering group elements after pseudomultiplication.}

The folklore that Kummer surfaces cannot be used in true group-based
cryptosystems may appear mathematically true---but it is algorithmically false.

Indeed, it has long been known that x-only arithmetic on elliptic curves
can be used for full scalar multiplication: 
L\'opez and Dahab~\cite{Lopez--Dahab} 
(followed by Okeya and Sakurai~\cite{Okeya--Sakurai}
and Brier and Joye~\cite{Brier--Joye})
showed that the auxiliary values computed by the
x-only Montgomery ladder can be used to recover the missing y-coordinate,
and hence to compute full scalar multiplications on 
elliptic curves.  

In this work we extend this technique from elliptic curves to genus~2,
and from one- to two-dimensional scalar multiplication.
In the abstract, our algorithms
follow the same common pattern:
\begin{enumerate}
    \item
        First, \(\Project\) the inputs (and possibly some auxiliary elements)
        from \(\G\) to \(\Gquotient\) using the \(x\) map;
    \item
        then pseudomultiply in \(\Gquotient\)
        (that is, compute \(x([m]P)\) or \(x([m]P\oplus[n]Q)\)
        using a differential addition chain);
    \item
        finally, efficiently \(\Recover\) the correct preimage \([m]P\) 
        (or \([m]P\oplus[n]Q\)) in \(\G\) 
        from the outputs of the pseudomultiplication.
\end{enumerate}
The one-dimensional version of this pattern in~\S\ref{sec:1D}
uses the well-known Montgomery ladder~\cite{petmon}
for its pseudomultiplication;
the two-dimensional version in~\S\ref{sec:2D}
is based on Bernstein's binary differential addition chain~\cite{DJB-chain}.

In \S\ref{sec:EC} we apply our algorithms to various models of elliptic
curves.  We recover Okeya--Sakurai and Brier--Joye multiplication, along
with new uniform, compact two-dimensional scalar multiplication
algorithms.

Moving to genus~2,
the Jacobian point recovery method we present in~\S\ref{sec:genus2}
solves the problem of uniform genus~2 arithmetic 
(at least for scalar multiplication):
rather than wrestling with the special cases of Cantor's algorithm,
we can descend to the faster, more uniform Kummer surface, pseudomultiply
there, and then recover the right Jacobian point afterwards.
Our methods work on the most general form of the Kummer, 
and are easy to adapt to more specialized models.

In \S\ref{sub:gaudry-mult} 
we specialize to the Gaudry model for Kummer surfaces, which have the
fastest known arithmetic.  The result is exceptionally fast one- and
two-dimensional scalar multiplication algorithms for the genus~2
Jacobians that admit these Kummer surfaces.
Finally, in \S\ref{sec:signatures} we give a concrete example 
of how our algorithms can be applied to create fast instances of
Schnorr signature schemes.

\begin{remark}
    \label{rem:Lubicz--Robert-1}
    Damien Robert has pointed out to us that 
    his recent preprint with David Lubicz~\cite{Lubicz--Robert}
    uses similar techniques to improve the efficiency of 
    their explicit arithmetic of 
    general abelian varieties in arbitrarily high dimension 
    based on theta functions.  
    See Remark~\ref{rem:Lubicz--Robert-2} below for further details.
\end{remark}

\subsubsection{Notation and conventions} 
Throughout, \(\FF_q\) denotes a finite field of characteristic \(> 3\).
As usual, we use ${\bf M}$, ${\bf S}$, ${\bf I}$ and ${\bf a}$ to
respectively denote the costs of one field multiplication, squaring,
inversion, and addition in \(\FF_q\). 
We also use ${\bf m}_c$ to denote the cost of a field multiplication by a fixed constant $c$, which is often significantly different to ${\bf M}$, i.e., when $c$ is a \emph{small} curve constant. 

In our algorithms, 
we use the notation \((x_1,\ldots,x_n) \gets (y_1,\ldots,y_n)\)
to denote \emph{parallel assignment}.
In sequential terms, 
\((x,y) \gets (z,w)\) is equivalent to \(x \gets z ; y \gets w\)
(and to \(y \gets w ; x \gets z\)),
while \((x,y) \gets (y,x)\), which swaps \(x\) and \(y\),
is equivalent to \(t \gets x ; x \gets y; y \gets t\),
where \(t\) is a temporary variable.

\section{
    Key subroutines
}
\label{sub:subroutines}

In this paper we work with various models of elliptic curves, 
Jacobians, and Kummer surfaces, but viewed from a high level the
algorithms are essentially the same; we can therefore make some
substantial simplifications by presenting them as template algorithms
acting on an abstract abelian group \(\G\),
and its quotient \(\G/\pm1\), by black-box subroutines
(whose implementation details we will provide later).
We therefore assume the existence of six algorithms 
acting on elements of \(\G\) and \(\Gquotient\):

\begin{enumerate}
    \item 
        \underline{\(\Project : \G \to \Gquotient\)}
        implements the mapping \(x : \G \to \Gquotient\);
        that is,
        \[
            \Project(P) = x(P)
            \quad
            \text{ for all }
            P \in \G
            \ .
        \]
        This is trivial in the elliptic context,
        where it amounts to dropping one of the coordinates,
        and only slightly less straightforward in genus~2.
    \item
        \underline{\(\xDBL : \Gquotient \to \Gquotient\)}
        implements pseudo-doubling in \(\Gquotient\):
        that is,
        \[
            \xDBL(x(P)) = x([2]P)
            \quad
            \text{ for all }
            P \in \G
            \ .
        \]
        Efficient formul\ae{} for \(\xDBL\) are known in each of our
        contexts.
    \item 
        \underline{\(\xADD : (\Gquotient)^3 \to \Gquotient\)} 
        implements the standard differential addition:
        \[
            \xADD(x(P),x(Q),x(P\ominus Q)) = x(P\oplus Q)
            \quad
            \text{ for all }
            P \not= Q \in \G
            \ .
        \]
        As with \(\xDBL\), efficient formul\ae{} for \(\xADD\) 
        are known in each of our contexts.
    \item
        \underline{\(\xDBLADD : (\Gquotient)^3\to(\Gquotient)^2 \)} 
        implements a simultaneous doubling and differential addition:
        \[
            \xDBLADD(x(P),x(Q),x(Q\ominus P)) = (x([2]P),x(P\oplus Q))
            \quad
            \text{ for all }
            P \not= Q \in \G
            \ .
        \]
        While \(\xDBLADD\) may be na\"ively defined by computing
        an \(\xDBL\) and an \(\xADD\) separately---that is, as 
        \[
            \xDBLADD(x(P),x(Q),x(Q\ominus P)) 
            =
            (\xDBL(x(P)),\xADD(x(P),x(Q),x(Q\ominus P)))
        \]
        ---sometimes in practice it can be implemented more efficiently 
        by exploiting shared intermediate operands,
        so we treat it as a distinct operation.
    \item
        \underline{\(\ADD:\G\times\G\to\G\)}
        computes the group law in \(\G\): 
        \[
            \ADD(P,Q) = P\oplus Q
            \quad
            \text{ for all }
            P, Q \in \G 
            \ .
        \]
        \(\ADD\) is only used once in the two-dimensional algorithm; 
        it is not used at all in the one-dimensional algorithm.  
        Minimizing \(\ADD\)s is a deliberate strategy:
        in practice, it is often a relatively costly 
        and potentially non-constant-time operation 
        compared with \(\xADD\) (especially in genus~2).
    \item
        \underline{\(\Recover:\G\times (\Gquotient)^2\to\Gquotient\)}
        computes preimages under \(x\): 
        \[
            \Recover(P,x(Q),x(Q\oplus P)) = Q
            \quad
            \text{ for all }
            P, Q \in \G \setminus \G[2]
            \ .
        \]
        This operation was introduced for binary elliptic curves
        by L\'opez and Dahab~\cite{Lopez--Dahab},
        for Montgomery models of elliptic curves 
        by Okeya and Sakurai~\cite{Okeya--Sakurai},
        and for short Weierstrass models of elliptic curves
        by Brier and Joye~\cite{Brier--Joye}.
        We extend these techniques to genus~2 in~\S\ref{sec:genus2}
        and~\S\ref{sub:gaudry-mult}.
\end{enumerate}

\section{
    Uniform one-dimensional scalar multiplication
}
\label{sec:1D}

Algorithm~\ref{algorithm:Ladder-template}
is a template for uniform one-dimensional scalar multiplication;
it is suitable for use anywhere in curve-based cryptosystems
where the calculation \((m,P) \mapsto [m]P\) is required,
but especially those where \(m\) is secret.
Algorithm~\ref{algorithm:Ladder-template}
applies the \Pattern{} pattern
to lift the \(x\)-only Montgomery ladder for
pseudomultiplication in \(\Gquotient\)
to a full scalar multiplication routine for \(\G\),
generalizing and abstracting the methods 
of~\cite{Lopez--Dahab}, \cite{Okeya--Sakurai}, and~\cite{Brier--Joye}.

\RestyleAlgo{boxed}\LinesNumbered
\begin{algorithm}
    \caption{A one-dimensional uniform scalar multiplication template,
        based on the Montgomery ladder}
    \label{algorithm:Ladder-template}
    \SetKwFunction{Uniform1}{Uniform1}
    \SetKwInOut{Input}{Input}
    \SetKwInOut{Output}{Output}

    \Input{A positive integer $m = \sum_{i=0}^{\beta-1}m_i2^i$,
        with \(m_{\beta-1} \not=0\), and an element $P$ of~$\G$
    }
    \Output{$R = [m]P$}

    $ x_P \gets \Project(P) $ \;
    $ (t_1,t_2) \gets (x_P, \xDBL(x_P)) $ \;
    \For{$i = \beta-2$ down to $0$}{
        \uIf{$m_i = 0$}{
            $ (t_1,t_2) \gets \xDBLADD(t_1,t_2,x_P) $ \;
        }
        \Else{
            $ (t_2,t_1) \gets \xDBLADD(t_2,t_1,x_P) $ \;
        }
    }
    $ R \gets \Recover(P,t_1,t_2) $ \;
    \Return{$R$} \;
\end{algorithm}

\begin{lemma}
    \label{lemma:Ladder-cost}
    Let \(m\) be a non-negative integer of bitlength \(\beta\),
    and \(P\) an element of~\(\G\).
    Algorithm~\ref{algorithm:Ladder-template} computes \([m]P\)
    using
    one call to \(\Project\),
    one call to \(\xDBL\),
    \(\beta-1\) calls to \(\xDBLADD\),
    and
    one call to \(\Recover\).
\end{lemma}
\begin{proof}
    Line~1 of
    Algorithm~\ref{algorithm:Ladder-template} 
    calls \(\Project\)
    to map \(P\) into \(\Gquotient\).
    Lines 2-9 are just the standard Montgomery ladder~\cite{petmon}.
    At the end of each of the \(\beta\) iterations of the loop
    (each of which calls \(\xDBLADD\) once),
    we have
    \[
        (t_1,t_2) 
        = 
        (x([\lfloor m/2^i \rfloor]P),x([\lfloor m/2^i \rfloor + 1]P))
        \ ;
    \]
    so at the end of the loop,
    at Line~10,
    \((t_1,t_2) = (x([m]P),x([m]P\oplus P))\).
    The \(\Recover(P,t_1,t_2)\) in Line~10
    therefore yields \([m]P\).
    \qed
\end{proof}

In its abstract form, 
Algorithm~\ref{algorithm:Ladder-template} is uniform and constant-time 
with respect to fixed-length~\(m\).
In practice, 
the implementation of \(\xDBLADD\) must also be uniform and constant-time.
If uniform, constant-time behaviour is required 
with respect to \(P\), 
then the implementation of \(\Project\) must also be uniform and constant-time.

\section{
    Uniform two-dimensional scalar multiplication
}
\label{sec:2D}

Algorithm~\ref{algorithm:DJB-template}
is a template for uniform two-dimensional scalar multiplication.
It is intended for use in cryptographic routines that require
computing \([m]P\oplus[b]Q\), 
especially when at least one of \(m\) and \(n\) are secret,
but it is also useful for implementing endomorphism-accelerated
scalar multiplications, 
which compute \([m]P\) as \([m_0]P \oplus [m_1]\phi(P) \).
Algorithm~\ref{algorithm:DJB-template}
applies the \Pattern{} pattern
to a pseudomultiplication based on
Bernstein's binary differential addition chain~\cite{DJB-chain}.
Our algorithm is similar to its \(x\)-only counterpart,
which was used (without any proof of correctness) in~\cite{CHS}.

\subsection{Bernstein's binary differential addition chain}

Bernstein defined his binary differential addition chain 
in~\cite[\S4]{DJB-chain} as follows.
First, set
\[
    C_0(0,0) = C_1(0,0) 
    := 
    \left(
        (0,0), (1,0), (0,1), (1,-1)
    \right)
    \ ;
\]
then the chain \(C_D(A,B)\) is defined recursively by
\[
    C_D(A,B) := C_d(a,b) \; || \; (O,E,M)
\]
where \(||\) denotes concatenation,
\(O\), \(E\), and \(M\) are defined by
\begin{align}
    \label{eq:O-def}
    O & := (A + (A+1\bmod{2}), B + (B+1\bmod{2}))
    \ ,
    \\
    \label{eq:E-def}
    E & := (A + (A+0\bmod{2}), B + (B+0\bmod{2}))
    \ ,
    \\
    \label{eq:M-def}
    M & := (A + (A + D\bmod{2}), B + (B + D + 1\bmod{2}))
    \ ,
\end{align}
and \(a\), \(b\), and \(d\) are defined by
\[
    a := \lfloor{A/2}\rfloor 
    \ ,
    \quad
    b := \lfloor{B/2}\rfloor
    \ ,
    \quad
    d := \begin{cases}
        D     & \text{if } a\equiv A \text{ and } b\equiv B \pmod{2} \ ,
        \\
        0     & \text{if } a\equiv A \text{ and } b\not\equiv B \pmod{2} \ ,
        \\
        1     & \text{if } a\not\equiv A \text{ and } b\equiv B \pmod{2} \ ,
        \\
        1 - D & \text{if } a\not\equiv A \text{ and } b\not\equiv B \pmod{2} \ .
    \end{cases}
\]
Observe that \(O\) contains two odd integers,
\(E\) two even integers,
and \(M\) is ``mixed'', 
with one even and one odd integer.
The differences \(M - O\), \(M - E\), and \(O - E\)
depend only on \(D\) and the parities of \(A\) and \(B\),
as shown in Table~\ref{tab:DJB-chain-differences}.

\begin{table}
    \caption{
        The differences between \(M\), \(O\), and \(E\)
        as functions of \(D\) 
        and \(A,B\pmod{2}\).
    }
    \label{tab:DJB-chain-differences}
    \begin{center}
        \begin{tabular}{c|c||c|c|c}
            \(A \pmod{2}\) & \(B \pmod{2}\) 
                & \(O - E\) & \(M - O\) & \(M - E\) 
            \\
            \hline
            \(0\) & \(0\) & \((1,1)\)   & \((D-1,-D)\)  & \((D,1-D)\)
            \\
            \(0\) & \(1\) & \((1,-1)\)  & \((D-1,D)\)   & \((D,D-1)\)
            \\
            \(1\) & \(0\) & \((-1,1)\)  & \((1-D,-D)\)  & \((-D,1-D)\)
            \\
            \(1\) & \(1\) & \((-1,-1)\) & \((1-D,D)\)   & \((-D,D-1)\)
            \\
            \hline
        \end{tabular}
    \end{center}
\end{table}

By definition, 
the triple \((O,E,M)\) contains three of the four pairs 
\( (A,B) \), \( (A+1,B) \), \( (A,B+1) \), and \( (A+1,B+1) \).
The missing pair is
\( (A + (A + D + 1 \bmod{2}), B + (B + D \bmod{2})) \),
from which it follows immediately that
\[
    C_{(A\bmod{2})}(A,B)
    \text{ contains } (A,B)
    \ .
\]

\begin{lemma}
    \label{lemma:DJB-chain-driver}
    With the notation above:
    Suppose \((a,b) \not= (0,0)\),
    and write \(o,e,m\) for the last three terms of \(C_d(a,b)\),
    so \(C_D(A,B) = (\ldots,o,e,m,O,E,M)\).
    Then \(O\), \(E\), and \(M\) 
    can be expressed in terms of \(o\), \(e\), \(m\), \(D\), and the parities of
    \(A\), \(B\), \(a\), and \(b\) 
    using the relations in Table~\ref{tab:DJB-chain-relations}.
\end{lemma}
\begin{proof}
    The result follows---after elementary but lengthy calculations---from
    the definition of \(C_D(A,B)\), considerations of parity,
    and the values in Table~\ref{tab:DJB-chain-differences}
    applied to \(o\), \(e\), and \(m\), with \(d\) derived from 
    the first four columns.
    \qed
\end{proof}

\begin{table}
    \caption{
        Relations between the adjacent triples \((o,e,m)\) and \((O,E,M)\).
    }
    \label{tab:DJB-chain-relations}
    \begin{center}
        \begin{tabular}{c|c|c|@{\; }c@{\; }||@{\; }cc|@{\; }c@{\; }|@{\; }cc}
            \(A - B\) & \(A - a\) & \(B - b\) 
            & \multirow{2}{*}{\(D\)} 
            & \multirow{2}{*}{\(O\)} & difference
            & \multirow{2}{*}{\(E\)} 
            & \multirow{2}{*}{\(M\)} & difference
            \\
            \(\pmod{2}\) & \(\pmod{2}\) & \(\pmod{2}\) & & & of summands
            & & & of summands
            \\
            \hline
            \hline
            0 & 0 & 0 & 0 & 
            \(o + e\) & \(\pm(1,1)\) & \(2e\) & \(m + e\) & \(\pm(0,1)\) 
            \\
            0 & 0 & 0 & 1 & 
            \(o + e\) & \(\pm(1,1)\) & \(2e\) & \(m + e\) & \(\pm(1,0)\) 
            \\
            0 & 0 & 1 & 0 & 
            \(o + e\) & \(\pm(1,1)\) & \(2m\) & \(m + e\) & \(\pm(0,1)\) 
            \\
            0 & 0 & 1 & 1 & 
            \(o + e\) & \(\pm(1,1)\) & \(2m\) & \(m + o\) & \(\pm(1,0)\) 
            \\
            0 & 1 & 0 & 0 & 
            \(o + e\) & \(\pm(1,1)\) & \(2m\) & \(m + o\) & \(\pm(0,1)\) 
            \\
            0 & 1 & 0 & 1 & 
            \(o + e\) & \(\pm(1,1)\) & \(2m\) & \(m + e\) & \(\pm(1,0)\) 
            \\
            0 & 1 & 1 & 0 & 
            \(o + e\) & \(\pm(1,1)\) & \(2o\) & \(m + o\) & \(\pm(0,1)\) 
            \\
            0 & 1 & 1 & 1 & 
            \(o + e\) & \(\pm(1,1)\) & \(2o\) & \(m + o\) & \(\pm(1,0)\) 
            \\
            1 & 0 & 0 & 0 & 
            \(o + e\) & \(\pm(1,-1)\) & \(2e\) & \(m + e\) & \(\pm(0,1)\) 
            \\
            1 & 0 & 0 & 1 & 
            \(o + e\) & \(\pm(1,-1)\) & \(2e\) & \(m + e\) & \(\pm(1,0)\) 
            \\
            1 & 0 & 1 & 0 & 
            \(o + e\) & \(\pm(1,-1)\) & \(2m\) & \(m + e\) & \(\pm(0,1)\) 
            \\
            1 & 0 & 1 & 1 & 
            \(o + e\) & \(\pm(1,-1)\) & \(2m\) & \(m + o\) & \(\pm(1,0)\) 
            \\
            1 & 1 & 0 & 0 & 
            \(o + e\) & \(\pm(1,-1)\) & \(2m\) & \(m + o\) & \(\pm(0,1)\) 
            \\
            1 & 1 & 0 & 1 & 
            \(o + e\) & \(\pm(1,-1)\) & \(2m\) & \(m + e\) & \(\pm(1,0)\) 
            \\
            1 & 1 & 1 & 0 & 
            \(o + e\) & \(\pm(1,-1)\) & \(2o\) & \(m + o\) & \(\pm(0,1)\) 
            \\
            1 & 1 & 1 & 1 & 
            \(o + e\) & \(\pm(1,-1)\) & \(2o\) & \(m + o\) & \(\pm(1,0)\) 
            \\
            \hline
        \end{tabular}
    \end{center}
\end{table}

\subsection{Encoding the chain}

Our aim is to use \(C_{m_0}(m,n)\) to compute \(x([m]P\oplus[n]Q)\),
where
\[
    m = \sum_{i=0}^{\beta-1} m_i2^i
    \quad
    \text{and}
    \quad
    n = \sum_{i=0}^{\beta-1} n_i2^i
    \ .
\]
Treating the \(m_i\) and \(n_i\) as bits, 
with \(\oplus\) denoting binary addition (xor)
and \(\otimes\) binary multiplication (and),
we define the sequence of \emph{transition vectors}
\[
    c_i 
    := 
    \left(
        (m_i \oplus n_i), 
        (m_i \oplus m_{i+1}), 
        (n_i \oplus n_{i+1}), 
        d_i
    \right)
    \text{ for }
    0 \le i < \beta-1
    \ ,
\]
where \( d_0 := m_0 \)
and
\[
    d_{i+1} := 
    ((d_{i}\oplus1)\otimes(m_i\oplus m_{i+1}))
    \oplus
    ( d_{i}\otimes (n_i\oplus n_{i+1}\oplus1))
    \text{ for } 
    i \ge 0
    \ .
\]
The coordinates of \(c_i\) correspond to the first four columns of
Table~\ref{tab:DJB-chain-relations},
with \(A = \lfloor m/2^{i} \rfloor\),
\(a = \lfloor m/2^{(i+1)} \rfloor\),
\(B = \lfloor n/2^{i} \rfloor\),
\(b = \lfloor n/2^{(i+1)} \rfloor\),
\(D = d_i\), and \(d = d_{i+1}\).

In view of Lemma~\ref{lemma:DJB-chain-driver},
given \(m\) and \(n\),
we can iteratively reconstruct the entire chain \(C_{m_0}(m,n)\)
from \(d_{\beta-1}\) and the sequence of transition vectors
\(c_0,\ldots,c_{\beta-2}\).
Algorithm~\ref{algorithm:DJB-chain-encoding}
computes precisely this data to encode \(C_{m_0}(m,n)\).
We note that the \(c_i\) also encode the difference elements
required for each differential addition.

\begin{algorithm}
    \caption{\texttt{ChainEncoding}: Computes and encodes Bernstein's 
    two-dimensional ``binary'' differential addition chain}
    \label{algorithm:DJB-chain-encoding}
    \SetKwFunction{ChainEncoding}{ChainEncoding}
    \SetKwInOut{Input}{Input}
    \SetKwInOut{Output}{Output}
    \Input{Positive $\beta$-bit integers 
        $m = \sum_{i=0}^{\beta-1} m_i2^i$ 
        and $n = \sum_{i=0}^{\beta-1} n_i2^i$.
    }
    \Output{A sequence \(C\) of \(\beta-1\) transition vectors
        and one additional bit \(d_{\beta-1}\),
        encoding the differential addition chain \(\mathcal{C}_{m_0}(m,n)\).
    }
    $C \gets [\ ]$ \;
    $d \gets m_0$ \;
    \For{$i \gets 0 $ up to $\beta-2$}{
        Append $(
            (m_i\oplus m_{i+1}),
            (n_i\oplus n_{i+1}),
            (m_{i+1}\oplus n_{i+1}),
            d
        )$ to $C$ \;
        $ 
            d
            \gets 
            ( (d\oplus1)\otimes(m_i\oplus m_{i+1}) )
            \oplus
            ( d\otimes (n_i\oplus n_{i+1}\oplus 1) )
        $ \;
    }
    \Return{$(C,d)$} \;
\end{algorithm}

\subsection{Two-dimensional scalar multiplication}

If we map each pair \((a,b)\) in \(C_{m_0}(m,n)\)
to the element \(x([a]P\oplus[b]Q)\) in \(\Gquotient\),
then \(C_{m_0}(m,n)\)
yields a method of computing \(x([m]P\oplus[n]Q)\)
using a sequence of \(\xADD\)s (and \(\xDBL\)s) with the fixed differences
\(x(P)\), \(x(Q)\), \(x(P\oplus Q)\), and \(x(P\ominus Q)\).
Thus, 
the sequence of transition vectors in our encoding of \(C_{m_0}(m,n)\)
gives us a natural iterative algorithm 
for computing \(x([m]P\oplus[n]Q)\)
starting from the fixed differences.
Using this as the core of our \Pattern{} pattern
yields
Algorithm~\ref{algorithm:DJB-template}, 
which computes \([m]P\oplus[n]Q\) in \(\G\).

\begin{algorithm}
    \caption{Two-dimensional uniform scalar multiplication template}
    \label{algorithm:DJB-template}
    \SetKwFunction{ChainEncoding}{ChainEncoding}
    \SetKwInOut{Input}{Input}
    \SetKwInOut{Output}{Output}
    \SetKwData{eveneven}{E}
    \SetKwData{oddodd}{O}
    \SetKwData{mixed}{M}
    \Input{Positive $\beta$-bit integers 
        $m = \sum_{i=0}^{\beta-1} m_i2^i$ 
        and $n = \sum_{i=0}^{\beta-1} n_i2^i$ 
        with $m_{\beta-1}$ and $n_{\beta-1}$ not both zero,
        and elements $P, Q$ of $\G$
    }
    \Output{$R = [m]P\oplus[n]Q$}
    \BlankLine
    $((c_0,\ldots,c_{\beta-2}),d_{\beta-1}) \gets \ChainEncoding{m,n}$ \;
    $S \gets \ADD(P,Q)$ \;
    $(x_P,x_Q,x_\oplus) \gets (\Project(P),\Project(Q),\Project(S))$ \;
    $x_\ominus \gets \xADD(x_P,x_Q,x_{\oplus})$ \;
    \Switch{$(m_{\beta-1},n_{\beta-1},d_{\beta-1})$}{
        \lCase{$(0,1,0):$}{
            $(\mixed,(\eveneven,\oddodd)) \gets (x_Q,\xDBLADD(x_Q,x_P,x_\ominus))$
        }
        \lCase{$(0,1,1):$}{
            $(\oddodd,(\eveneven,\mixed)) \gets (x_\oplus,\xDBLADD(x_Q,x_\oplus,x_P))$
        }
        \lCase{$(1,0,0):$}{
            $(\oddodd,(\eveneven,\mixed)) \gets (x_\oplus,\xDBLADD(x_P,x_\oplus,x_Q))$
        }
        \lCase{$(1,0,1):$}{
            $(\mixed,(\eveneven,\oddodd)) \gets (x_P,\xDBLADD(x_P,x_Q,x_\ominus))$
        }
        \lCase{$(1,1,0):$}{
            $(\oddodd,(\eveneven,\mixed)) \gets (x_\oplus,\xDBLADD(x_\oplus,x_P,x_Q))$
        }
        \lCase{$(1,1,1):$}{
            $(\oddodd,(\eveneven,\mixed)) \gets (x_\oplus,\xDBLADD(x_\oplus,x_Q,x_P))$
        }
    }
    \For{$i \gets \beta-2$ down to $0$}{
        \Switch{$c_i$}{
            \lCase{$(0,0,0,0):$}{ 
                $(\oddodd,(\eveneven,\mixed)) \gets
                (\xADD(\oddodd,\eveneven,x_\oplus),\xDBLADD(\eveneven,\mixed,x_Q))$
            }
            \lCase{$(0,0,0,1):$}{ 
                $(\oddodd,(\eveneven,\mixed)) \gets
                (\xADD(\oddodd,\eveneven,x_\oplus),\xDBLADD(\eveneven,\mixed,x_P))$
            }
            \lCase{$(0,0,1,0):$}{ 
                $(\oddodd,(\eveneven,\mixed)) \gets 
                (\xADD(\oddodd,\eveneven,x_\oplus),\xDBLADD(\mixed,\eveneven,x_Q))$
            }
            \lCase{$(0,0,1,1):$}{ 
                $(\oddodd,(\eveneven,\mixed)) \gets
                (\xADD(\oddodd,\eveneven,x_\oplus),\xDBLADD(\mixed,\oddodd,x_P))$
            }
            \lCase{$(0,1,0,0):$}{ 
                $(\oddodd,(\eveneven,\mixed)) \gets
                (\xADD(\oddodd,\eveneven,x_\oplus),\xDBLADD(\mixed,\oddodd,x_Q))$
            }
            \lCase{$(0,1,0,1):$}{ 
                $(\oddodd,(\eveneven,\mixed)) \gets
                (\xADD(\oddodd,\eveneven,x_\oplus),\xDBLADD(\mixed,\eveneven,x_P))$
            }
            \lCase{$(0,1,1,0):$}{ 
                $(\oddodd,(\eveneven,\mixed)) \gets
                (\xADD(\oddodd,\eveneven,x_\oplus),\xDBLADD(\oddodd,\mixed,x_Q))$
            }
            \lCase{$(0,1,1,1):$}{ 
                $(\oddodd,(\eveneven,\mixed)) \gets
                (\xADD(\oddodd,\eveneven,x_\oplus),\xDBLADD(\oddodd,\mixed,x_P))$
            }
            \lCase{$(1,0,0,0):$}{ 
                $(\oddodd,(\eveneven,\mixed)) \gets
                (\xADD(\oddodd,\eveneven,x_\ominus),\xDBLADD(\eveneven,\mixed,x_Q))$
            }
            \lCase{$(1,0,0,1):$}{ 
                $(\oddodd,(\eveneven,\mixed)) \gets
                (\xADD(\oddodd,\eveneven,x_\ominus),\xDBLADD(\eveneven,\mixed,x_P))$
            }
            \lCase{$(1,0,1,0):$}{ 
                $(\oddodd,(\eveneven,\mixed)) \gets
                (\xADD(\oddodd,\eveneven,x_\ominus),\xDBLADD(\mixed,\eveneven,x_Q))$
            }
            \lCase{$(1,0,1,1):$}{ 
                $(\oddodd,(\eveneven,\mixed)) \gets
                (\xADD(\oddodd,\eveneven,x_\ominus),\xDBLADD(\mixed,\oddodd,x_P))$
            }
            \lCase{$(1,1,0,0):$}{ 
                $(\oddodd,(\eveneven,\mixed)) \gets
                (\xADD(\oddodd,\eveneven,x_\ominus),\xDBLADD(\mixed,\oddodd,x_Q))$
            }
            \lCase{$(1,1,0,1):$}{ 
                $(\oddodd,(\eveneven,\mixed)) \gets
                (\xADD(\oddodd,\eveneven,x_\ominus),\xDBLADD(\mixed,\eveneven,x_P))$
            }
            \lCase{$(1,1,1,0):$}{ 
                $(\oddodd,(\eveneven,\mixed)) \gets
                (\xADD(\oddodd,\eveneven,x_\ominus),\xDBLADD(\oddodd,\mixed,x_Q))$
            }
            \lCase{$(1,1,1,1):$}{ 
                $(\oddodd,(\eveneven,\mixed)) \gets
                (\xADD(\oddodd,\eveneven,x_\ominus),\xDBLADD(\oddodd,\mixed,x_P))$
            }
        }
    }
    \Switch{$(m_0,n_0)$}{
        \lCase{$(0,0):$}{ $R \gets \Recover(S,\eveneven,\oddodd)$  }
        \lCase{$(0,1):$}{ $R \gets \Recover(P,\mixed,\oddodd)$  }
        \lCase{$(1,0):$}{ $R \gets \Recover(P,\mixed,\eveneven)$  }
        \lCase{$(1,1):$}{ $R \gets \Recover(S,\oddodd,\eveneven)$  }
    }
    \Return{$R$} \;
\end{algorithm}

\begin{theorem}
    \label{theorem:DJB-cost}
    Let \(P\) and \(Q\) be elements of \(\G\),
    let \(m\) and \(n\) be positive integers,
    and let \(\beta\) be the bitlength of \(\max(m,n)\).
    Algorithm~\ref{algorithm:DJB-template} computes \([m]P\oplus[n]Q\)
    using one call to \({\tt ADD}\),
    three calls to \(\Project\),
    \(\beta\) calls to each of \(\xADD\) and \(\xDBLADD\),
    and one call to \(\Recover\). 
\end{theorem}
\begin{proof}
    Algorithm~\ref{algorithm:DJB-template}
    begins by calling Algorithm~\ref{algorithm:DJB-chain-encoding}
    to compute
    the sequence of transition vectors and \(d_{\beta-1}\) 
    for the given \(m\) and \(n\).
    Lines~2--4
    compute the required differences \(x_P := x(P)\), \(x_Q := x(Q)\), 
    \(x_\oplus := x(P\oplus Q)\), and
    \(x_\ominus := x(P\ominus Q)\)
    by computing \(S := P\oplus Q\) 
    with the single call to \(\ADD\), 
    then using the three calls to \(\Project\)
    to compute \(x_P\), \(x_Q\), and \(x_\oplus = x(S)\),
    before \(x_\ominus = \xADD(x_P,x_Q,x_\oplus)\).

    Throughout the rest of the algorithm
    \(O\) corresponds to the odd-odd pair,
    \(E\) to the even-even pair,
    and \(M\) to the mixed pair in a segment of \(C_{m_0}(m,n)\),
    where each pair \((a,b)\) is associated with \(x([a]P\oplus[b]Q)\).
    Lines~5--12 use a single \(\xDBLADD\) 
    to initialize \(O\), \(E\), and \(M\)
    such that \(C_{d_{\beta-1}}(m_{\beta-1},n_{\beta-1}) = C_0(0,0)||(O,E,M)\).
    Lines~13--31 iterate over the sequence of transition vectors 
    to compute the last three elements of \(C_{m_0}(m,n)\).
    After each iteration,
    the triple \((O,E,M)\) satisfies
    \[
        C_{d_i}(\lfloor{m/2^{i}}\rfloor,\lfloor{n/2^{i}}\rfloor)
        = 
        C_{d_{i+1}}(\lfloor{m/2^{i+1}}\rfloor,\lfloor{n/2^{i+1}}\rfloor)
        \ ||\ 
        (O,E,M)
        \ .
    \]
    Each iteration requires
    two differential additions and a pseudodoubling,
    but we observe that the pseudodoubling always applies 
    to one of the arguments of a differential addition,
    so each of the \(\beta-1\) iterations 
    requires exactly one \(\xADD\) and one \(\xDBLADD\).

    After the loop is completed, at Line~33, 
    Eqs.~\eqref{eq:O-def}, \eqref{eq:E-def}, and~\eqref{eq:M-def}
    imply
    \[
        O = x(R\oplus\Delta_O)\ ,
        \quad
        E = x(R\oplus\Delta_E)\ ,
        \quad
        \text{and}
        \quad
        M = x(R\oplus\Delta_M)\ ,
    \]
    where \(R = [m]P\oplus[n]Q\) 
    and \(\Delta_O\), \(\Delta_E\), and \(\Delta_M\)
    are given in the following table:
    \begin{center}
        \begin{tabular}[c]{c||c|c|c|c||@{\;}l}
            \((m_0,n_0)\) & \(\Delta_O\) & \(\Delta_E\) 
            & \(\Delta_M\) if \(d_0=0\) & \(\Delta_M\) if \(d_0=1\)
            & \([m]P\oplus[n]Q\)
            \\
            \hline
            \((0, 0)\) & \(P\oplus Q\) & \(0_\G\) & \(Q\) & \(P\) 
            & \(\Recover(P\oplus Q, E, O)\)
            \\
            \((0, 1)\) & \(P\) & \(Q\) & \(0_\G\) & \(P\oplus Q\) 
            &
            \(\Recover(P, M, O)\)
            \\
            \((1, 0)\) & \(Q\) & \(P\) & \(P\oplus Q\) & \(0_\G\) 
            & \(\Recover(P, M, E)\)
            \\
            \((1, 1)\) & \(0_\G\) & \(P\oplus Q\) & \(P\) & \(P\oplus Q\) 
            & \(\Recover(P\oplus Q, O, E)\)
            \\
            \hline
        \end{tabular}
    \end{center}
    In each case, we can recover \([m]P\oplus[n]Q\)
    by applying \(\Recover\) with the arguments specified by the last
    column of the corresponding row.
    This is precisely what is done in Lines 33-38;
    Line~39 then returns the result, \([m]P\oplus[n]Q\).
    \qed
\end{proof}

Like Algorithm~\ref{algorithm:Ladder-template},
Algorithm~\ref{algorithm:DJB-template}
is uniform and constant-time in its abstract form,
at least for fixed-length multiscalars \((m,n)\).
In practice,
for Algorithm~\ref{algorithm:DJB-template} to be uniform and constant-time 
with respect to~\(m\) and~\(n\),
the implementations of \(\xADD\) and \(\xDBLADD\) 
must be uniform and constant-time.
For uniform and constant-time behaviour 
with respect to \(P\) and \(Q\),
the implementations of \(\ADD\) and \(\Project\) 
must also be uniform and constant-time.

\begin{remark}
    The core of Algorithm~\ref{algorithm:DJB-template}
    is similar to the algorithm in~\cite[App.~C]{CHS}, 
    but with a different (and slightly simpler) encoding of the addition chain.
    Here, the \(i\)-th transition vector is
    \( ((m_i\oplus n_i), (m_i\oplus m_{i+1}), (n_i\oplus n_{i+1}), d_i) \);
    in~\cite[App.~C]{CHS} it is
    \( ((m_i\oplus m_{i+1}\oplus n_i\oplus n_{i+1}), (m_i\oplus m_{i+1}), (m_i\oplus n_i), d_i) \),
    and the order of the vectors is reversed (as is the order of the
    loop iterations).
\end{remark}

\begin{remark}
    Implementers should notice that every adjacent pair 
    of \(\xADD\) and \(\xDBLADD\) operations 
    in Algorithm~\ref{algorithm:DJB-template}
    share one operand in their differential additions.
    Further savings might therefore be made 
    by sharing a few intermediate calculations 
    between the \(\xADD\) and \(\xDBLADD\) calls---that is, by
    implicitly defining an \(\mathtt{xDBLADDADD}\) operation, 
    as in~\cite[App.~C]{CHS}.
    We have not done this here, for two reasons.
    First, it somewhat obscures the fundamental symmetry of the addition chain. 
    Second, looking ahead at the explicit formul\ae{} for ${\tt xADD}$ and ${\tt xDBLADD}$
    suggests that there are no intermediate calculations that can be shared 
    in the elliptic curve scenarios. 
    However, close inspection of the fast Kummer arithmetic in~\S\ref{sub:fastkumarith} 
    reveals that two adjacent differential additions sharing a common summand 
    can be merged to save 8 field additions. 
    It is unclear to us that whether this potential saving would give a
    net benefit after the code complexity required to perform the merging,
    so we have not included this potential saving in our operation counts.
\end{remark}

\begin{remark}
    These techniques should readily extend to the
    higher-dimensional Montgomery-like differential addition chains
    described by Brown~\cite{Brown}.  
    We do not investigate this here.
\end{remark}

\section{
    Efficient scalar multiplication on elliptic curves
}
\label{sec:EC}

We now pass from the abstract to the concrete.
Applying the one-dimensional pattern to elliptic curves,
we recover Okeya--Sakurai-style multiplication algorithms;
applying the two-dimensional pattern yields something new.

\subsection{Montgomery models}

The most obvious application of our methods is to 
elliptic curves with Montgomery models
\[
    BY^2Z = X(X^2 + AXZ + Z^2) \ ,
\]
which are defined to optimize the \(\xADD\), \(\xDBL\), and \(\xDBLADD\)
operations.
Indeed, historically, the first appearance of the 
\Pattern{} pattern
was in the context of one-dimensional scalar multiplication on
Montgomery models~\cite{Okeya--Sakurai}.
Important examples of Montgomery curves include 
Curve25519,
recently recommended for standardization by the CFRG.

In this context,
\( \Project: (x,y) \mapsto x \)
is trivially computed.
The \(\ADD\), \(\xADD\), \(\xDBL\), and \(\xDBLADD\) operations
are all presented---and thoroughly costed---in the EFD~\cite{EFD}.
The \(\Recover\) operation was derived by Okeya and Sakurai in~\cite[\S3]{Okeya--Sakurai}. 
The operation counts for all of these operations 
are compiled in Table~\ref{table:Montgomery-operations}.
\begin{table}
    \caption{
        Costs of operations for projective Montgomery models 
        \(By^2 = x(x^2 + Ax + 1)\), 
        where ${\bf m}_A$ denotes multiplications by $A$ and $(A+2)/4$, 
        and \(\mathbf{m}_B\) denotes multiplications by \(B\).
        A point is normalized if its projective \(Z\)-coordinate
        has been scaled to 1.
    }
    \label{table:Montgomery-operations}
    \begin{center}
        \begin{tabular}{r|c|c|c|c|c|c|@{\;}l}
            Algorithm    & 
           \ {\bf M} \ & \ {\bf S} \ & \ ${\bf m}_A$ \ & \ ${\bf m}_B$ \ & \, {\bf a} \, & \, {\bf I} \, & Conditions
            \\
            \hline
            \( {\tt ADD}\) & 1 & 1 & 0 & 0 & 5 & 1 & \(P\) and \(Q\)
            normalized
            \\
            \(\Project\) & 0 & 0 & 0 & 0 & 0 & 0  &  ---
            \\
            \(\xDBL\)    & 2 & 5 & 1 & 0 & 0 & 0 &---
            \\
            \(\xADD\)    & 4 & 2 & 0 &0 & 6 & 0 & ---
            \\
            \(\xADD\)    & 3 & 2 & 0 &0 & 6 &  0& \(P\ominus Q\) normalized
            \\
            \(\xDBLADD\) & 6 & 4 & 1 & 0 & 8 & 0 & ---
            \\
            \(\xDBLADD\) & 5 & 4 & 1 & 0 & 8 & 0 & \(P\ominus Q\) normalized 
            \\
            \(\Recover\) & 13 & 1 & 1 &1 & 8 & 1 & \(P\) normalized
            \\
            \hline
        \end{tabular}
    \end{center}
\end{table}

We now apply the \Pattern{} pattern
using the routines above.
In the one-dimensional case,
we recover the Okeya--Sakurai algorithm~\cite{Okeya--Sakurai}.
\begin{theorem}
    \label{theorem:Montgomery}
    Let \(\EC/\FF_q\) be an elliptic curve in Montgomery form. 
    \begin{enumerate}
        \item
            Let \(P\) be a point in
            \(\EC(\FF_q)\setminus\EC[2]\),
            and let \(m\) be a positive \(\beta\)-bit integer.
            Then Algorithm~\ref{algorithm:Ladder-template} computes
            \([m]P\)
            in 
            \[
                \FqM{(5\beta+10)}
                +
                \FqS{(4\beta+2)}
                +
                (\beta+1)\mathbf{m}_A
                +
                1\mathbf{m}_B
                +
                (8\beta+6)\mathbf{a}
                +
                \FqI{1}
                \ . 
            \]
        \item
            Let \(P\) and \( Q\) be points in
            \(\EC(\FF_q)\setminus\EC[2]\),
            and let \(m\) and \(n\) be positive \(\beta\)-bit integers. 
            Then Algorithm~\ref{algorithm:DJB-template} 
            computes \([m]P\oplus[n]Q\)
            in 
            \[
                \FqM{(8\beta+14)}
                +
                \FqS{(6\beta+2)}
                +
                (\beta+1)\mathbf{m}_A
                +
                1\mathbf{m}_B
                +
                (14\beta+13)\mathbf{a}
                +
                \FqI{2}
                \ .
            \]
    \end{enumerate}
\end{theorem}

\subsection{Edwards models}

As a byproduct of Theorem~\ref{theorem:Montgomery}, we obtain new scalar multiplication algorithms for Edwards models~\cite{edwards}. Indeed, every twisted Edwards model can
be transformed into a Montgomery model with a linear change of
coordinates, so one option to perform scalar multiplications on Edwards curves is to pass to and from an associated Montgomery model, making use of the operations summarized in Table~\ref{table:Montgomery-operations}. Important examples of Edwards models in practice include Ed25519~\cite{EdDSA} and Goldilocks~\cite{goldilocks}. 

Another option here is to exploit pseudomultiplication formul\ae{} that
are native to Edwards curves, as described by Gaudry and Lubicz~\cite[\S
6.2]{gaudry-lubicz}. In this case, for Edwards curves $\EC/\FF_q \colon
x^2+y^2=1+dx^2y^2$ with $d=r^2$ and $r \in \FF_q$, we define ${\tt
Project}: (x,y) \mapsto y$ and use the formul\ae{} in the EFD~\cite{EFD} to compute the pseudomultiplication via Algorithm~\ref{algorithm:Ladder-template}. 

The \(\Recover \colon (P,y(Q),y(Q \oplus P)) \mapsto Q=(x(Q),y(Q)) \) operation is defined by rearranging the Edwards addition law between $P$ and $Q$, to make \(x(Q)\) the subject, i.e., 
\[
    x(Q) = \frac{y(Q \oplus P)-y(Q)y(P)}{x(P)\left(dy(P)y(Q)y(Q \oplus P)-1\right)}
    \ . 
\]

\subsection{Scalar multiplication on short Weierstrass models}

Not every elliptic curve has a Montgomery or Edwards model;
the most general form for an elliptic curve over \(\FF_q\)
in characteristic greater than 3
is the short Weierstrass model
\[
    \EC : y^2 = x^3 + Ax + B \subset \PP^2 \ .
\] 
Important examples of curves commonly implemented as Weierstrass models 
include the NIST~\cite{NIST} and Brainpool curves~\cite{Brainpool}.

Our key subroutines are implemented as follows.
For brevity, we use affine coordinates here, but the resulting
formul\ae{} are easily projectivized: see Brier and Joye or the EFD.
In this context, 
\(\Project: P = (x,y) \mapsto x\) is trivially computed.
The \(\ADD\), \(\xADD\), and \(\xDBL\), and \(\xDBLADD\) operations
are all specified as efficient straight-line programs 
in the EFD~\cite{EFD}.
Brier and Joye describe~\(\Recover\) for short Weierstrass models
in~\cite[Prop.~3]{Brier--Joye}:
it maps \((P,x(Q),x(P\oplus Q))\) to \( Q = (x(Q),y(Q)) \), where
\[
    y(Q)
    =
    \frac{
        2B + (A + x(P)x(Q))(x(P) + x(Q)) - x({P\oplus Q})(x(P) - x(Q))^2
    }{
        2y(P)
    }
    \ .
\]
Table~\ref{table:short-Weierstrass-operations} summarizes the operation
counts for all of this routines.

Applying Algorithm~\ref{algorithm:Ladder-template} with these subroutines
yields the scalar multiplication algorithm for
Weierstrass models in~\cite{Brier--Joye}.
Applying Algorithm~\ref{algorithm:DJB-template} with these subroutines
yields a new algorithm for two-dimensional scalar multiplication
on Weierstrass curves.

\begin{table}
    \caption{
        Costs of operations for projective short Weierstrass models \(y^2=x^3+ax+b\), where ${\bf m}_b$ denotes multiplications by $b$, $2b$ and $4b$. 
        A point is normalized if its projective \(Z\)-coordinate
        has been scaled to 1.
    }
    \label{table:short-Weierstrass-operations}
    \begin{center}
        \begin{tabular}{r@{\;}|@{\;}c@{\;}|@{\;}c@{\;}|@{\;}c@{\;}|@{\;}c@{\;}|@{\;}c@{\;}|@{\;}c@{\;}|@{\;}l}
            Algorithm    & 
            {\bf M} & {\bf S} & ${\bf m}_a$ & ${\bf m}_b$ & {\bf a} & {\bf I} & Conditions
            \\
            \hline
            \( {\tt ADD}\) & 1 & 1 & 0 & 0 & 4 & 1 & \(P\) and \(Q\) normalized
            \\
            \(\Project\) & 0 & 0 & 0 & 0 & 0 & 0&--- 
            \\
            \(\xDBL\)    & 2 & 5 & 1 & 2 & 8 &0 &---
            \\
            \(\xADD\)    & 7 & 2 & 1 & 1 & 6 &0& ---
            \\
            \(\xADD\)    & 6 & 2 & 1 & 1 & 4 &0& \(P\ominus Q\) normalized
            \\
            \(\xDBLADD\) & 9 & 7 & 2 & 3 & 12 & 0& ---
            \\
            \(\xDBLADD\) & 8 & 7 & 2 & 3 & 12 &0& \(P\ominus Q\) normalized
            \\
            \(\Recover\) & 11 & 2 & 1 & 1 & 7 & 1 & \(P\) normalized
            \\
            \hline
        \end{tabular}
    \end{center}
\end{table}

\begin{theorem}
    Let \(\EC/\FF_q\) be an elliptic curve in short Weierstrass form.
    \begin{enumerate}
        \item
            Let \(P\) be a point in
            \(\EC(\FF_q)\setminus\EC[2]\),
            and let \(m\) be a positive \(\beta\)-bit integer.
            Then Algorithm~\ref{algorithm:Ladder-template} computes
            \([m]P\)
            in
            \[
                \FqM{(8\beta+5)}
                +
                \FqS{7\beta}
                +
                2\beta\mathbf{m}_a
                +
                3\beta\mathbf{m}_b
                +
                (12\beta+3)\mathbf{a}
                +
		        \FqI{1}
                \ .
            \]
        \item
            Let \(P\) and \( Q\) be points in
            \(\EC(\FF_q)\setminus\EC[2]\),
            and let \(m\) and \(n\) be positive \(\beta\)-bit integers. 
            Then Algorithm~\ref{algorithm:DJB-template} 
            computes \([n]P\oplus[n]Q\)
            in 
            \[
                \FqM{(14\beta+12)}
                +
                \FqS{(9\beta+3)}
                +
                (3\beta+1)\mathbf{m}_a
                +
                (4\beta+1)\mathbf{m}_b
                +
                (16\beta+11)\mathbf{a}
                +
                \FqI{2}
                \ .
            \]
    \end{enumerate}
\end{theorem}
\begin{proof}
    Take the values from 
    Table~\ref{table:short-Weierstrass-operations}
    in Lemma~\ref{lemma:Ladder-cost} (for the first part)
    and Theorem~\ref{theorem:DJB-cost} (for the second).
    \qed
\end{proof}

\section{
    Genus~2 Jacobians and General Kummer Surfaces
}
\label{sec:genus2}

We now turn our focus to genus~2 cryptosystems,
where \(\G\) is (a subgroup of) the Jacobian \(\Jac{\C}\)
of a genus~2 curve
\[
    \C: y^2 = f(x) = f_6x^6 + f_5x^5 + \cdots + f_1x + f_0
    \quad
    \text{over } 
    \FF_q
    \ .
\] 
Elements of \(\Jac{\C}(\FF_q)\) 
are presented in their standard Mumford representation:
\[
    P \in \Jac{\C}(\FF_q)
    \longleftrightarrow
    \langle a(x) = x^2 + a_1x + a_0, b(x) = b_1x + b_0 \rangle
\]
where \(a_1\), \(a_0\), \(b_1\), and \(b_0\) are in \(\FF_q\) and
\[
    b(x)^2 \equiv f(x) \pmod{a(x)} 
    \ .
\]
The quotient \(\Gquotient\) is (a subset of) a Kummer surface,
which is a quartic surface in \(\PP^3\).
There are two main ways to represent the Kummer:
the ``general'' model \(\Kgen{\C}\)
(see~Cassels and Flynn~\cite[Ch.~3]{casselsflynn})
and the ``fast'' model \(\Kfast{\C}\)
algorithmically developed by Chudnovsky and
Chudnovsky~\cite{Chudnovsky--Chudnovsky},
and introduced in cryptography by Gaudry~\cite{Gaudry}.

Broadly speaking, 
the fast Kummer model corresponds to 
the Montgomery model for elliptic curves, 
while the general model corresponds to the Weierstrass model. 
Indeed,
as in the elliptic curve situation,
fast Kummers offer significant gains in performance and uniformity
(indeed, they are at the heart of a number of record-breaking
Diffie--Hellman implementations),
but at the price of a lot of rational \(2\)-torsion:
hence, not every Kummer can be put in fast form.

We will define the \(\Project\) and \(\Recover\) operations
for general Kummers;
since any Kummer surface can be transformed into a general model over
the ground field, these \(\Project\) and \(\Recover\) routines
can easily be specialized to fast Kummers (or to any other interesting
models of Kummer surfaces).
We do not lose much in taking this approach compared with deriving 
a new, specialized~\(\Project\) and~\(\Recover\) 
for Jacobians with fast Kummers from scratch, 
because 
Algorithms~\ref{algorithm:Ladder-template} and~\ref{algorithm:DJB-template}
only call \(\Project\) a few times,
and \(\Recover\) once.
Scalar pseudomultiplication on general Kummers is relatively slow,
and has held little interest for cryptographers thus far;
we briefly discuss their arithmetic and pseudomultiplication
in~\S\ref{sec:Kgen-scalar-multiplication}.

\subsection{Genus~2 arithmetic and side-channel attacks}
\label{sec:g2-side-channel}

The group law on \(\Jac{\C}\) (and hence the \(\ADD\) operation) 
is typically computed using Cantor's
algorithm, specialized to genus~2. 
This style of arithmetic has a serious drawback in cryptographic
contexts: it is highly susceptibile to simple side-channel attacks. 
Similar to the textbook addition for short Weierstrass models of
elliptic curves, Cantor's algorithm in genus~2 treats several input
cases differently, \emph{branching} off into distinct explicit
computations. Explicit formul\ae{} that are derived for generic additions
fail to compute correctly when one or both of the inputs are
\emph{special} points---that is, points in \(\Jac{\C}\) 
where $a(x)=x-\alpha$ or $a(x)=(x-\alpha)^2$. 
While such special points are sparse enough in \(\Jac{\C}\) 
that random scalar multiplications do not encounter them, 
they are plentiful enough that attackers could easily mount 
exceptional procedure attacks~\cite{izu-takagi}, 
forcing legitimate users into special cases and using timing variabilities to recover secret data. 

The \Pattern{} approach detailed in~\S\ref{scalars-general-kummer} addresses this problem. 
Like $x$-only (pseudo)scalar multiplication on elliptic curves, 
the explicit formul\ae{} for doublings and differential additions on
certain Kummer surfaces associated to \(\Jac{\C}\) are well behaved on
all inputs, including the images of special points in \(\Jac{\C}\) under
the \(\Project\) map. The only time special points might be encountered
is at the two end points of the scalar multiplication; but suitable
point validation can detect and reject special points as inputs, and
special outputs in \(\Jac{\C}\) are not the goal of an exceptional
procedure attack---adversaries can only hope to trigger exceptional
points early on in a scalar multiplication, where the number of possible intermediate points is subexponential.

\subsection{Point recovery on the general Kummer}\label{scalars-general-kummer}

The general model of the Kummer surface $\Kgen{\C}$ associated to
$\Jac{\C}$ is defined by
\begin{equation}
    \label{eq:Kgen-equation}
    \Kgen{\C} : 
        K_2(\xi_1,\xi_2,\xi_3)\xi_4^2
        + 
        K_1(\xi_1,\xi_2,\xi_3)\xi_4 
        + 
        K_0(\xi_1,\xi_2,\xi_3)
        = 
        0
        \ ,
\end{equation}
where
\begin{align}
K_2 &= \xi_2^2-4\xi_1\xi_3, \nonumber\\
K_1 &=-2(2f_0\xi_1^3+f_1\xi_1^2\xi_2+2f_2\xi_1^2\xi_3+f_3\xi_1\xi_2\xi_3 + 2f_4\xi_1\xi_3^2+f_5\xi_2\xi_3^2+2f_6\xi_3^3) , \nonumber\\
K_0 &= (f_1^2-4f_0f_2)\xi_1^4-4f_0f_3\xi_1^3\xi_2-2f_1f_3\xi_1^3\xi_3-4f_0f_4\xi_1^2\xi_2^2\nonumber\\
&\quad +4(f_0f_5-f_1f_4)\xi_1^2\xi_2\xi_3+(f_3^2+2f_1f_5-4f_2f_4-4f_0f_6)\xi_1^2\xi_3^2-4f_0f_5\xi_1\xi_2^3\nonumber\\
&\quad +4(2f_0f_6-f_1f_5)\xi_1\xi_2^2\xi_3+4(f_1f_6-f_2f_5)\xi_1\xi_2\xi_3^2-2f_3f_5\xi_1\xi_3^3-4f_0f_6\xi_2^4\nonumber\\
&\quad -4f_1f_6\xi_2^3\xi_3-4f_2f_6\xi_2^2\xi_3^2-4f_3f_6\xi_2\xi_3^3+(f_5^2-4f_4f_6)\xi_3^4. \nonumber
\end{align}

\subsection{Projection from Jacobians to general Kummers}
\(\Project\) implements the map $\Jac{\C} \rightarrow \Kgen{\C}$ 
described in~\cite[Eqs.~(3.1.3--5)]{casselsflynn};
it maps generic points \(\langle x^2+a_1x+a_0,b_1x+b_0 \rangle\)
in \(\Jac{\C}\)
to $(\xi_1 \colon \xi_2 \colon \xi_3 \colon \xi_4)$ in \(\Kgen{\C}\), 
where 
\begin{align}
    \label{eq:Kgen-Project-1}
    &
    \xi_1 = 1 \ ,\quad 
    \xi_2 = -a_1 \ ,\quad 
    \xi_3 = a_0 \ ,\quad \text{and}
    \\
    \label{eq:Kgen-Project-2}
    &
    \xi_4
    =
    b_1^2+(a_1^2-a_0)\left(f_5a_1-f_6(a_1^2-a_0)\right)+a_1(f_3-f_4a_1)-f_2
    \ .
\end{align}
This costs $\FqM{5}+\FqS{2}+ \Fqa{7}$ (saving $\FqM{1}$ if $f_6=0$).

\subsection{Recovering Jacobian points from general Kummers} 
We now derive explicit formul\ae{} for
\[
    \Recover: (P, x(Q),x(Q\oplus P)) \longmapsto Q 
\]
in genus~2.
We follow the approach of Okeya--Sakurai~\cite{Okeya--Sakurai}
and Brier--Joye~\cite{Brier--Joye} 
for elliptic curves,
rewriting the equations used for computing $x(Q\oplus P)$ 
in terms of $P=(x(P),y(P))$ and $Q=(x(Q),y(Q))$, making $y(Q)$ the unknown. 
We may suppose that \(P\) and \(Q\) are nonzero and not points of order~2.
 
Suppose we are given\footnote{
    We suppose that the inputs $x(Q)$, $x(Q\oplus P)$, and $x(Q\ominus P)$ 
    are normalized here, to simplify the exposition. 
    In practice, we use projectivized forms of these formul\ae{}
    (which corresponds to replacing \(\xi_2\), \(\xi_3\), and \(\xi_4\)
    with \(\xi_2/\xi_1\), \(\xi_3/\xi_1\), and \(\xi_4/\xi_1\), etc.),
    in order to handle the non-normalized inputs that we encounter at
    the end of our addition chains.
    The input point $P$ in $\Jac{\C}$ can be presumed to be normalized, 
    since it is the input to the scalar multiplication routine. 
}
\begin{align*}
    P = {} &
    \langle x^2+a_1(P)x+a_0(P), b_1(P)x+b_0(P) \rangle \in \Jac{\C}(\FF_q) \ ,
    \\
    x(Q)= {}
    & (1 : \xi_2 : \xi_3 : \xi_4) \in \Kgen{\C}(\FF_q) \ ,
    \\
    x(Q\oplus P)= {}
    & (1:\xi^\oplus_2:\xi^\oplus_3:\xi^\oplus_4) \in \Kgen{\C}(\FF_q) \ ;
\end{align*}
we want to \(\Recover\) 
\[
    Q = \langle x^2+a_1(Q)x+a_0(Q), b_1(Q)x+b_0(Q) \rangle \in \Jac{\C}(\FF_q) \ .
\]
We already have $a_1(Q)=-\xi_2$ and $a_0(Q)=\xi_3$ 
from Eq.~\eqref{eq:Kgen-Project-1};
it remains to compute $b_1(Q)$ and $b_0(Q)$. 

Okeya and Sakurai noticed that the formul\ae{} for $y$-coordinate recovery
on Montgomery curves are simpler if $x(Q\ominus P)$ is also known
(see~\cite[pp. 129--130]{Okeya--Sakurai}); 
we observed the same simplification in genus~2, 
where it was evident that 
computing $x(Q \ominus P)$ from $x(Q)$, $x(P)$ and $x(Q \oplus P)$ 
(using one more differential addition on the corresponding Kummer)
yielded a faster overall recovery. 
We therefore begin our \(\Recover\) with a call to \(\xADD\),
to compute 
\[
    x(Q\ominus P)
    =
    (1 : \xi^\ominus_2 : \xi^\ominus_3 : \xi^\ominus_4)
    = 
    \xADD(Q,P,x(Q\oplus P))
    \ .
\]

Since \(P\) and \(Q\) are nonzero,
they correspond to unique degree-2 divisors
on $\C$ (cf.~\cite[Ch.~1]{casselsflynn}):
\[
    P \longleftrightarrow [(u_P,v_P)+(u_P',v_P')]
    \ ,
    \quad 
    Q \longleftrightarrow [(u_Q,v_Q)+(u_Q',v_Q')]
    \ .
\]
We do not compute
the values of \(u_P,v_P,u_P',v_P',u_Q,v_Q,u_Q'\), and \(v_Q'\) 
(which are generally defined over a quadratic extension):
here they simply serve as formal devices, to aid our derivation of 
recovery formul\ae{}.
Let
\begin{align*}
    G_1 & := E_1 + E_2 
    \ , &
    G_2 & := u_P' E_1 + u_P E_2
    \ , \\
    G_3 & := E_3 + E_4
    \ , &
    G_4 & := u_Q'E_3 + u_QE_4
    \ ,
\end{align*}
where
\begin{align*}
    E_1 & = \frac{v_P}{(u_P-u_P')(u_P-u_Q)(u_P-u_Q')} 
    \ , &
    E_2 & = \frac{v_P'}{(u_P'-u_P)(u_P'-u_Q)(u_P'-u_Q')}
    \ , \\
    E_3 & = \frac{v_Q}{(u_Q-u_P')(u_Q-u_P)(u_Q-u_Q')} 
    \ , &
    E_4 & = \frac{v_Q'}{(u_Q'-u_P)(u_Q'-u_P')(u_Q'-u_Q)}
    \ .
\end{align*}
The \(G_i\) are functions of \(P\) and \(Q\),
because they are symmetric with respect to 
\((u_P,v_P)\leftrightarrow(u_P',v_P')\) and \((u_Q,v_Q)\leftrightarrow(u_Q',v_Q')\);
below, 
we will compute them as functions of \(P\), \(x(Q\oplus P)\), and
\(x(Q\ominus P)\).
For notational convenience, we define
\begin{align}
    \label{eq:Z_1Z_2}
    Z_1 & := -(\xi_2 + a_1(P))
    \ , 
    &
    Z_2 & := \xi_3 - a_0(P)
    \ , 
    \\
    \label{eq:Z_3Z_4}
    Z_3 & := a_1(P) \xi_3 + a_0(P)\xi_2 
    \ , 
    & 
    Z_4 & := \xi_3Z_2 - \xi_2Z_3
    \ ,
\end{align}
\begin{equation}
    \label{eq:Ddef}
    D := Z_2^2 - Z_1 Z_3
    \ ,
\end{equation}
\begin{equation}
    \label{eq:Deltadef}
    \Delta 
    := 
    -4G_2^2 
    + 2(\xi_2^\oplus + \xi_2^\ominus) G_1G_2
    - 2(\xi_3^\oplus + \xi_3^\ominus) G_1^2
    \ ,
\end{equation}
and 
\begin{align}
    T := {} & f_6 - G_1^2 - G_3^2 
    \nonumber
    \\
    \label{eq:Tdef}
    = {} & 
    f_6 - G_1^2 -
    \frac{1}{D^2}\left(
        \begin{array}{c}
            \xi_4 D + f_0 Z_1^2 - f_1 Z_1 Z_2 + f_2 Z_2^2 - f_3 Z_2 Z_3 
            \\
            + f_4 Z_3^2 - f_5 Z_3 Z_4 + f_6 Z_4^2
        \end{array}
    \right)
    \ .
\end{align}

We now use the fact that the cubic polynomial
\begin{align*}
    l(x) 
    = & \
    E_1(x-u_P')(x-u_Q)(x-u_Q') + E_2(x-u_P)(x-u_Q)(x-u_Q') 
    \\
    & {} + E_3(x-u_P)(x-u_P')(x-u_Q') + E_4(x-u_P)(x-u_P')(x-u_Q)
    \\
    = & \
    (G_1x - G_2)(x^2 + a_1(Q)x + a_0(Q)) 
    +
    (G_3x - G_4)(x^2 - \xi_2x + \xi_3) 
\end{align*}
satisfies $\ell(x) \equiv b(x) \bmod{a(x)}$ 
when $\langle a(x) , b(x)\rangle$ is the Mumford representation 
of $P$, $Q$ or $\ominus(P\oplus Q)$
(this is just the geometric definition of the group law on~\(\Jac{\C}\);
the cubic \(\ell(x)\) is analogous to the line through \(P\), \(Q\), and
\(\ominus(P\oplus Q)\) in the classic elliptic curve group law).
Together with $b(x)^2 \equiv f(x) \pmod{a(x)}$, 
which is satisfied by the Mumford representation \(\langle a(x),
b(x)\rangle\) of every point
on~\(\Jac{\C}\),
this gives the relations
\begin{equation}
    \label{eq:Kgen-Recover-1}
    \begin{pmatrix}
        b_1(Q)
        \\ b_0(Q)
    \end{pmatrix}
    = 
    \begin{pmatrix}
        \xi_2Z_1 + Z_2 \ \ \ & Z_1
        \\
        \xi_3Z_1 & Z_2
    \end{pmatrix}
    \begin{pmatrix}
        G_3
        \\ 
        G_4
    \end{pmatrix}
    \ ,
\end{equation}
\begin{equation}
    \label{eq:Kgen-Recover-2}
    \begin{pmatrix}
        G_3
        \\ 
        G_4
    \end{pmatrix}
    = 
    \frac{T}{\Delta}
    \begin{pmatrix}
        \xi_3^\ominus-\xi_3^\oplus \ \  \ \ & \xi_2^\oplus - \xi_2^\ominus 
        \\
        \xi_2^\oplus \xi_3^\ominus-\xi_2^\ominus \xi_3^\oplus  \ \  \ \ & \xi_3^\oplus - \xi_3^\ominus
    \end{pmatrix}
    \begin{pmatrix}
        G_1
        \\ 
        G_2
    \end{pmatrix}
    \ ,
\end{equation}
and
\begin{equation}
    \label{eq:Kgen-Recover-3}
    \begin{pmatrix}
        G_1
        \\ 
        G_2
    \end{pmatrix}
    = 
    \frac{1}{D}
    \begin{pmatrix}
        Z_2 & -Z_1
        \\
        -a_0(P)Z_1 & \ \  a_1(P)Z_1 - Z_2
    \end{pmatrix}
    \begin{pmatrix}
        b_1(P)
        \\ 
        b_0(P)
    \end{pmatrix}
    \ .
\end{equation}

Our \(\Recover\) operation is now defined as follows:
first, 
we use Eqs.~\eqref{eq:Z_1Z_2} and~\eqref{eq:Z_3Z_4} to compute the~\(Z_i\);
then Eq.~\eqref{eq:Ddef} yields \(D\).
Then we can use Eq.~\eqref{eq:Kgen-Recover-3}
to obtain $G_1$ and $G_2$,
which we use to compute \(\Delta\) and \(T\) 
using Eqs.~\eqref{eq:Deltadef} and~\eqref{eq:Tdef}.
Equation~\eqref{eq:Kgen-Recover-2}
then yields \(G_3\) and \(G_4\),
which we substitute into Eq.~\eqref{eq:Kgen-Recover-1}
to compute \(b_1(Q)\) and \(b_0(Q)\).
We have thus recovered the full point 
\[
    Q = \langle x^2 + a_1(Q)x + a_0(Q), b_1(Q)x + b_0(Q) \rangle
\] 
in \(\Jac{\C}\).
Altogether, 
computing the map $(P, x(Q), x(Q\oplus P), x(Q\ominus P)) \mapsto Q$ 
costs $\FqM{71}+\FqS{8}+\Fqc{8}+\Fqa{35}+\FqI{1}$;
adding
the cost of computing $x(Q\ominus P)$ via an \(\xADD\) operation,
we obtain the full cost of computing \(\Recover(P,x(Q),x(Q\oplus P))\).

\begin{remark}
    It is worth noting that $\xi_4^\oplus$ and $\xi_4^\ominus$ 
    do not appear in our recovery procedure; 
    this may be useful in scenarios where it is advantageous 
    to omit their computation or transmission.
\end{remark}

\subsection{Scalar multiplication on Jacobians via general Kummers}
\label{sec:Kgen-scalar-multiplication}

The use of general Kummers in cryptography was investigated by Smart and
Siksek~\cite{smart-siksek} and Duquesne~\cite{Duquesne}.
While they represent a natural generalization of \(x\)-only arithmetic
for elliptic curves, 
in their full generality they do not offer competitive performance.
The \(\xADD\) and \(\xDBL\) operations 
are defined by complicated biquadratic forms in the \(\xi_i\)
(see~\cite{casselsflynn}),
which are hard to evaluate quickly for general curve parameters.
While these formul\ae{} are completely compatible with our \Pattern{}
pattern, and yield scalar multiplication algorithms that may be
useful to number theorists, the cryptographic applications of these
algorithms are \emph{\`a priori} limited.
We will therefore skip any investigation of these scalar
multiplication algorithms here, and move on directly to 
Gaudry's fast Kummer surfaces.

We remark, nevertheless, that the use of these biquadratic forms in
conjunction with the \(\Project\) and \(\Recover\) defined above
yields a solution to the problem of defining uniform and constant-time
scalar multiplication algorithms for general genus~2 Jacobians.
We leave the eventual optimization of these algorithms as an open
problem, along with their application to Jacobians with special
endomorphism structure, or Jacobians whose Kummers can be put into 
intermediate forms between the general and fast models.

\section{
    Efficient Scalar Multiplication via Fast Kummers
}
\label{sub:gaudry-mult}

While the performance of general Kummer models is somewhat disappointing,
Gaudry~\cite{Gaudry} showed if we allow a certain \(2\)-torsion
structure on the Jacobian, then an alternative classical model
for the Kummer (investigated algorithmically by Chudnovsky and
Chudnovsky~\cite{Chudnovsky--Chudnovsky}) yields a dramatic speedup,
competitive with---and regularly outperforming---elliptic curve arithmetic.
Jacobians equipped with these ``fast'' Kummers 
are ideal candidates for our techniques;
here we use the model with squared theta coordinates
described in~\cite[Ch.~4]{cosset}.

\subsection{Construction of fast Kummers}
\label{sub:Kfast-construction}

Suppose we have \(a\), \(b\),
\(c\), \(d\), \(e\), and \(f\) in \(\FF_q\setminus\{0\}\) such that
\begin{align*}
    A & := a + b + c + d \ ,
    &
    B & := a + b - c - d \ ,
    \\
    C & := a - b + c - d \ ,
    &
    D & := a - b - c + d
\end{align*}
are nonzero, and
\[
    e = \frac{1+\alpha}{1-\alpha}  f
    \qquad
    \text{where}
    \qquad
    \alpha^2 = \frac{CD}{AB}.
\]

We set
\begin{align*}
    \lambda & = \frac{ac}{bd} 
    \ ,
    &
    \mu & = \frac{ce}{df} 
    \ ,
    &
    \nu & = \frac{ae}{bf} 
    \ ,
\end{align*}
and define an associated genus~2 curve \(\C\) 
in Rosenhain form:
\[
    \C : y^2 = f(x) = x(x-1)(x-\lambda)(x-\mu)(x-\nu)
    \ .
\]
The \emph{fast Kummer surface} for \(\C\) is
the quartic surface \(\Kfast{\C}\subset\PP^3\) 
defined by
\begin{equation}
    \label{eq:Kfast-equation}
    \Kfast{\C}:
    \left(
        \begin{array}{c}
            (X^2+Y^2+Z^2+T^2)
            \\
            -F (XT+YZ)-G(XZ+YT)-H(XY+ZT)
        \end{array}
    \right)^2
    = E XYZT 
\end{equation}
where
\begin{align*}
    E & := 4 abcd 
    \left(ABCD/((ad - bc)(ac-bd)(ab - cd))\right)^2
    \ ,
    \\
    F & := 
    (a^2 - b^2 - c^2 + d^2)/(ad-bc)
    \ ,
    \\
    G & := 
    (a^2 - b^2 + c^2 - d^2)/(ac-bd)
    \ ,
    \\
    H & := 
    (a^2 + b^2 - c^2 - d^2)/(ab-cd)
    \ .
\end{align*}

For efficient arithmetic on \(\Kfast{\C}\), 
we precompute the \emph{theta constants}
\begin{align*}
    x_0 & := 1 \ ,
    &
    y_0 & := a/b \ ,
    &
    z_0 & := a/c \ ,
    &
    t_0 & := a/d \ ,
    \\
    x_0' & := 1 \ ,
    &
    y_0' & := A/B \ ,
    &
    z_0' & := A/C \ ,
    &
    t_0' & := A/D \ .
\end{align*}
The image of the identity element of \(\Jac{\C}\)
in \(\Kfast{\C}\) is
\[
    (a\colon b \colon c \colon d) = (1/x_0\colon 1/y_0\colon 1/z_0\colon 1/d_0) \ ;
\]
we also observe that 
\((A:B:C:D) = (1/x_0':1/y_0':1/z_0':1/d_0')\).

\subsection{Projection from Jacobians to fast Kummers}

\(\Project\) implements the map $\Jac{\C} \rightarrow \Kfast{\C}$ 
from~\cite[\S 5.3]{cosset}, 
which is defined for generic points in \(\Jac{\C}\)
by 
\(
    \langle x^2+a_1x+a_0, b_1x+b_0 \rangle 
    \mapsto 
    (X \colon Y \colon Z \colon T)
\), 
where 
\begin{align*}
    X & = a \left(a_0(\mu-a_0)(\lambda+a_1+\nu)  - b_0^2\right) \ , \\
    Y & = b \left(a_0(\nu\lambda-a_0)(1+a_1+\mu) - b_0^2\right) \ , \\
    Z & = c \left(a_0(\nu-a_0)(\lambda+a_1+\mu)  - b_0^2\right) \ , \\
    T & = d \left(a_0(\mu\lambda-a_0)(1+a_1+\nu) - b_0^2\right) \ .
\end{align*}
This costs \(\FqM{11} + \FqS{1}+ \Fqc{3}+ \Fqa{12} + \FqI{1} \), 
assuming the output point is normalized with $X=1$. 

For special points \(\langle x-\alpha, \beta \rangle\) in \(\Jac{\C}\), 
\(\Project\) is defined by first adding a point of order 2 in
\(\Jac{\C}\), e.g., adding \(\langle x-\lambda, 0 \rangle\) to get
\(
    \langle 
        x^2-(\alpha+\beta)x+\alpha\beta
        ,
        \frac{\beta}{\alpha-\lambda}x-\frac{\beta\lambda}{\alpha-\lambda}
    \rangle
\), then using the above map to $\Kfast{\C}$, where the translation is
undone by using the coordinate permutation-and-negation that corresponds
to translation by \(\langle x-\lambda, 0 \rangle\) (see~\cite[\S3.4]{Gaudry}).
Finally, \(0_{\Jac{\C}}\) 
(whose Mumford representation is \(\langle 1, 0 \rangle\))
maps to \((a:b:c:d)\).

\subsection{Basic fast Kummer arithmetic}\label{sub:fastkumarith}

We recall the formul\ae{} and operation counts
for \(\xDBL\) and \(\xADD\) on fast Kummers from~\cite[\S3.2]{Gaudry}.
If
\begin{align*}
    x(P) & = (X_1:Y_1:Z_1:T_1) \ ,
    \\
    x(Q) & = (X_2:Y_2:Z_2:T_2) \ ,
    \\ 
    x(P\ominus Q) & = (X_\ominus:Y_\ominus:Z_\ominus:T_\ominus) \ ,
\end{align*}
then \(\xADD\) maps \((x(P),x(Q),x(P\ominus Q))\) to  \( (X_\oplus:Y_\oplus:Z_\oplus:T_\oplus) \),
where
\begin{align*}
    X_\oplus & = (X' + Y' + Z' + Y')^2/X_\ominus \ ,
    &
    Y_\oplus & = (X' + Y' - Z' - Y')^2/Y_\ominus \ ,
    \\
    Z_\oplus & = (X' - Y' + Z' - Y')^2/Z_\ominus \ ,
    &
    T_\oplus & = (X' - Y' - Z' + Y')^2/T_\ominus \ ,
\end{align*}
where 
\begin{align*}
    X' & = 
    x_0'(X_1 + Y_1 + Z_1 + T_1)(X_2 + Y_2 + Z_2 + T_2) \ ,
    \\
    Y' & = 
    y_0'(X_1 + Y_1 - Z_1 - T_1)(X_2 + Y_2 - Z_2 - T_2) \ ,
    \\
    Z' & = 
    z_0'(X_1 - Y_1 + Z_1 - T_1)(X_2 - Y_2 + Z_2 - T_2) \ ,
    \\
    T' & = 
    t_0'(X_1 - Y_1 - Z_1 + T_1)(X_2 - Y_2 - Z_2 + T_2)
    \ .
\end{align*}
We can compute \(\xADD\) using 8 squares, 7 products (3 by constants,
ignoring the multiplication by \(x_0' = 1\)), 
and 4 divisions (or, projectively, another 10 products); 
but in our applications, 
\((X_\ominus:Y_\ominus:Z_\ominus:T_\ominus)\) is fixed.
so we can precompute \(1/X_\ominus\),
\(1/Y_\ominus\), \(1/Z_\ominus\), and \(1/T_\ominus\), 
and then the 4 divisions become 4 products.

The \(\xDBL\) operation is defined by
exactly the same formul\ae{}
on setting \((X_2:Y_2:Z_2:T_2) = (X_1:Y_1:Z_1:T_1)\)
and \((X_\ominus:Y_\ominus:Z_\ominus:T_\ominus) = (a:b:c:d)\)
(so \(X' = x_0'(X_1 + Y_1 + Z_1 + T_1)^2\), and so on).

The combined \(\xDBLADD\) operation is outlined in Gaudry~\cite[\S3.3]{Gaudry}:
we share the computation of \(x_0'(X_1 + Y_1 + Z_1 + T_1)^2\),
etc.,
between the calculation of \(\xDBL\) and \(\xADD\).
The resulting operation costs 16 products and 1 square.

Finally, the function ${\tt ADD} \colon P, Q \rightarrow x(Q\oplus P)$
in $\Jac{\C}$ is computed via the formul\ae{} in~\cite[Eq. (12)]{jac-on-jac} at a cost of $\FqM{22}+\FqS{2}+\FqI{1}+\Fqa{27}$.

\subsection{Recovering Jacobian points from fast Kummers} The $\Recover$
operation for fast Kummers uses the \(\Recover\) for general Kummers
(defined in~\S\ref{scalars-general-kummer}) as a subroutine.
To move between the
fast Kummer $\Kfast{\C}$ and the general Kummer $\Kgen{\C}$, 
we modify
the map from $\Jac{\C}$ to $\Kfast{\C}$ given in~\cite[\S 5.3]{cosset}
(which was in turn derived from~\cite{vanwamelen}). 
This gives a linear isomorphism
\[
    \tau : \Kgen{\C} \longrightarrow \Kfast{\C} 
    \ ,
\]
defined on generic points by 
\[
    \tau : 
    (\xi_1 \colon \xi_2 \colon \xi_3 \colon \xi_4) 
    \longmapsto 
    (X \colon Y \colon Z \colon T)
    = 
    (\xi_1 \colon \xi_2 \colon \xi_3 \colon \xi_4) M(a,b,c,d)^t
    \ ,
\]
where 
\[
    M(a,b,c,d)
    = 
    \left( 
        \begin{array}{c@{\;\;}c@{\;\;}c@{\;\;}c}
            \mu(\lambda+\nu) & \nu\lambda(1+\mu) 
            & \nu(\lambda+\mu) & \mu\lambda(1+\nu) 
            \\
            -\mu & -\nu\lambda & -\nu & -\mu\lambda 
            \\
            \mu+1 & \lambda+\nu & \nu+1 & \lambda+\mu 
            \\
            -1 & -1 & -1 & -1
        \end{array} 
    \right) 
    \ .
\] 
The inverse map \( \tau^{-1} : \Kfast{\C} \to \Kgen{\C} \)
is defined by the matrix \((M(a,b,c,d)^t)^{-1}\).
Since evaluating these linear transformations
only involves multiplications by constants, 
evaluating \(\tau\) to map a point from \(\Kgen{\C}\) to \(\Kfast{\C}\)
(or \(\tau^{-1}\) to map a point from \(\Kfast{\C}\) back to \(\Kgen{\C}\))
costs $\FqM{16}+ \Fqa{12}$. 
Since the coordinates on both $\Kgen{\C}$ and $\Kfast{\C}$ are projective, 
we can save $\FqM{1}$ 
by scaling one of the entries in the transformation matrix to 1.  

With $x \colon \Jac{\C} \rightarrow K$, the operation $\Recover \colon
P, x(Q),x(Q\oplus P) \mapsto Q$ is computed as follows. As was done
in~\S\ref{scalars-general-kummer}, the first step is to compute $x(Q
\ominus P)$ from $x(P)$, $x(Q)$ and $x(Q \oplus P)$ via an \(\xADD\)
operation (note that $x(P)$ was already obtained during the \(\Project\)
operation). We then use $\tau^{-1}$ to map all four elements, $x(P)$,
$x(Q)$, $x(Q \oplus P)$ and $x(Q \ominus P)$ to the corresponding
general Kummer $\Kgen{\C}$. We can then use the \(\Recover\) operation
on $\Kgen{\C}$ to output a normalized point $Q$ in Mumford coordinates using $\FqM{71}+\Fqc{4}+\FqS{8}+\Fqa{34}+\FqI{1}$. Recall from~\S\ref{scalars-general-kummer} that the fourth Kummer coordinate $\xi_4$ is only needed for the Kummer point corresponding to $Q$. Thus, the map $\tau^{-1}$ is only computed in full once (costing $\FqM{15}+ \Fqa{12}$), but for the three other points can omit the fourth coordinate (costing $\FqM{11}+\Fqa{9}$). 

In total, including the initial \(\xADD\) operation (costing
$\FqM{14}+\FqS{4}+\Fqc{3}+\Fqa{12}$), the map $\Recover \colon P, x(Q),x(Q\oplus
P) \mapsto Q$ when $x \colon \Jac{\C} \rightarrow \Kfast{\C}$ costs $\FqM{133}+\FqS{12}+\Fqc{7}+\Fqa{97}+\FqI{1}$. Here we count multiplications by the curve constants $f_i$ as full multiplications. 

\subsection{Scalar multiplication on Jacobians with fast Kummers}

The costs of our key subroutines for fast Kummers are summarized 
in Table~\ref{table:Gaudry-operations}.

\begin{table}
    \caption{
          Costs of operations for fast Kummers, where \({\bf m}_c\) is used to denote multiplications by the theta constants ($x_0$, $x_0'$, etc) or by the curve constants ($f_i$). 
          Normalization refers to the respective projective coordinates.
    }
    \label{table:Gaudry-operations}
    \begin{center}
        \begin{tabular}{r@{\;}|@{\;}c@{\;}|@{\;}c@{\;}|@{\;}c@{\;}|@{\;}c@{\;}|@{\;}c@{\;}|@{\;}l}
            Algorithm    & 
            \ {\bf M} \ & \ {\bf S} \ & \({\bf m}_c\)  &\  {\bf a} & \ {\bf I} \ & Conditions
            \\
            \hline
            \({\tt ADD}\) & 22 & 2 & 0 & 27  & 1 & \(P\) and \(Q\) normalized
            \\
            \(\Project\) & 11 & 1 & 3 & 12  & 1&--- 
            \\
            \(\xDBL\)    & 0 & 8 & 6 & 16  & 0& --- 
            \\
            \(\xADD\)    & 14 & 4 & 3 & 24  & 0& --- 
            \\
            \(\xADD\)    & 7 & 4 & 3 & 24  & 0& \ $x(Q\ominus P)$ fixed
            \\
            \(\xDBLADD\) & 17 & 9 & 6 & 32  &  0&--- 
            \\
            \(\xDBLADD\) & 10 & 9 & 6 & 32  & 0& \ $x(Q\ominus P)$ fixed
            \\
            \(\Recover\) & 133 & 12 & 7 & 97  & 1& \ $P$ normalized 
            \\
            \hline
        \end{tabular}
    \end{center}
\end{table}

\begin{theorem}\label{thm:gaudrycost}
    Let \(\Jac{\C}\) be the Jacobian of a genus~2 curve admitting a fast
    Kummer surface, as in~\S\ref{sub:Kfast-construction}.
    \begin{enumerate}
        \item
            Let \(P\) be a point in
            \(\Jac{\C}(\FF_q)\setminus\Jac{\C}[2]\),
            and let \(m\) be a positive \(\beta\)-bit integer.
            Then Algorithm~\ref{algorithm:Ladder-template} computes
            \([m]P\)
            in 
            \[
                \FqM{(10\beta+134)}
                +
                \FqS{(9\beta+12)}
                +
                (6\beta+10)\mathbf{m}_c
                +
                (32\beta+93)\mathbf{a}
                +
                \FqI{2}
                \ .
            \]
        \item
            Let \(P\) and \( Q\) be points in
            \(\Jac{\C}(\FF_q)\setminus\Jac{\C}[2]\),
            and let \(m\) and \(n\) be positive \(\beta\)-bit integers. 
            Then Algorithm~\ref{algorithm:DJB-template} computes
            \([m]P\oplus[n]Q\)
            in 
            \[
                \FqM{(17\beta+194)}
                +
                \FqS{(13\beta+17)}
                +
                (9\beta+16)\mathbf{m}_c
                +
                (56\beta+160)\mathbf{a}
                +
                \FqI{2}
                \ .
            \]
    \end{enumerate}
\end{theorem}
\begin{proof}
    Take the values from 
    Table~\ref{table:Gaudry-operations}
    in Lemma~\ref{lemma:Ladder-cost} for the first part
    and Theorem~\ref{theorem:DJB-cost} for the second, using a simultaneous inversion~\cite{petmon} to replace the $\FqI{3}$ (corresponding to the 3 \(\Project\) operations) by $\FqM{6}+\FqI{1}$.
    \qed
\end{proof}

\begin{remark}
    \label{rem:Lubicz--Robert-2}
    As we noted in Remark~\ref{rem:Lubicz--Robert-1},
    similar techniques appear in Lubicz and Robert~\cite{Lubicz--Robert}
    for general abelian varieties in higher dimension embedded in
    projective space via theta functions.
    After suitable precomputations, 
    their ``compatible addition'' operation
    views the results of Montgomery-like pseudomultiplication
    algorithms on a Kummer variety \(K\) of dimension \(g\)
    as points on the corresponding abelian variety \(A\) 
    embedded in \(K\times K\),
    and therefore chooses the ``correct'' result of an addition on~\(K\).
    They explain how to map the resulting point into the \(4\Theta\)
    embedding of \(A\) 
    (in \(\PP^{4^g-1}\); for genus \(2\), this is \(\PP^{15}\)).
    A major difference with our treatment here is that 
    Robert and Lubicz cannot treat \(A\) as a Jacobian
    (since general abelian varieties of dimension \(g > 3\) are not Jacobians);
    hence, when specializing to genus~2,
    there is no connection with any curve \(\C\) or Jacobian
    \(\Jac{\C}\)
    (and in particular, the starting and finishing points do not involve the
    Mumford representation).
    Kohel~\cite{Kohel}
    explores similar ideas for elliptic curves,
    leading to a very interesting interpretation of Edwards curve arithmetic.
\end{remark}

\section{
    Applications to signatures
}
\label{sec:signatures}

Smart and Siksek~\cite{smart-siksek} first showed that 
the action of $\ZZ$ on \(\Gquotient\) 
could be used to instantiate Diffie--Hellman 
on \(\Gquotient\) rather than \(\G\).
In~\cite[\S 5]{smart-siksek}, they pondered whether protocols like
ElGamal encryption could be instantiated on \(\Gquotient\), 
and observed that the main obstruction appeared to be that 
such protocols require an addition in the group---that is, 
a true group structure and not a mere $\ZZ$-action.

Our \Pattern{} technique allows a range of cryptographic protocols 
beyond Diffie--Hellman key exchange to take advantage of the fast and
uniform (multi-)scalar multiplications offered on \(\Gquotient\), where
\(\G\) can now be any elliptic curve \emph{or} any genus~2 Jacobian
(defined over a large characteristic field). 

In this section we illustrate this by showing how Schnorr
signatures~\cite{schnorr} can take advantage of the fast arithmetic
available when \(\Gquotient\) is the fast Kummer surface $\Kfast{\C}$.
More specifically,
we describe an instantiation of a Schnorr signature scheme 
where \(\G\) is the Jacobian of the Gaudry--Schost curve~\cite{gaudry-schost} 
used recently to set Diffie--Hellman speed records~\cite{danjabe-and-nok} at the 128-bit security level. 
Following the design of EdDSA signatures~\cite{EdDSA}, 
we hash both the message $M$ and signer's public key $Q$ 
alongside the first half of the signature. 
We emphasise, however, that 
both of these choices are merely for illustrative purposes, 
and that any known variant of ElGamal signatures 
(Schnorr or otherwise) can use \(\Gquotient\) 
for any suitable choice of \(\G\). 

\subsection{The Gaudry-Schost Jacobian}
Let $q=2^{127}-1$,
choose an \(\alpha\) in \(\FF_q\)
such that \(363\alpha^2+833=0\),
define the six constants 
\[
    a = 11
    \ ,\quad
    b = -22
    \ ,\quad
    c = -19
    \ ,\quad
    c = -3
    \ ,\quad
    e = 1+\alpha
    \ ,\quad
    f = 1-\alpha
    \ ;
\]
setting \( \lambda := (ac)/(bd) \),
\( \mu := (ce)/(df) \),
and
\( \nu := (ae)/(bf) \),
we
let \(\C\) be the genus~2 curve over \(\FF_q\) defined by
\[
    \C : y^2=x(x-1)(x-\lambda)(x-\mu)(x-\nu)
    \ .
\] 
The Jacobian $\Jac{\C}$ of $\C$ has cardinality $\#\Jac{\C}(\FF_q)=2^4N$, 
where 
\[
    N=2^{250} - \mathtt{0x334D69820C75294D2C27FC9F9A154FF47730B4B840C05BD}
\]
is a 250-bit prime.

We define the 127-bit encoding of $\FF_q$ 
to be the little-endian encoding of $\{0,1,\dots, 2^{127}-2\}$,
and let $\mathcal{H} \colon \{0,1\}^* \rightarrow \{0,1\}^{512}$ 
be any suitable hash function with a 512-bit output
(eg.~SHA-512 or SHA3-512). 

To ensure that all our scalars are positive and of fixed bitlength,
when computing scalar multiplications \([m]P\) on \(\Jac{\C}\)
we systematically replace \(m\) with \(m' := (m\bmod{N})+3N\);
then \(m'\) always has exactly 252 bits (since both \(3N\) and \(4N\) have 252
bits).
Following the lead of~\cite{EdDSA}, 
the cofactor 16 is included in the key generation and verification operations 
to void any potential threat of small subgroup attacks~\cite{lim-lee}.
In this case, the scalar $16z$ is parsed by appending four zeroes to
the parsing of $z$ described above, so a multiplication by the 252-bit
scalar $z$ is followed by 4 doubling operations. Where applicable, the operation counts stated below include the cost of these 4 additional ${\tt DBL}$ or ${\tt xDBL}$ operations. 

To compare our \Pattern{} approach to traditional arithmetic in \(\Jac{\C}\), we take the operation counts from~\cite[Table 2]{jac-on-jac}, where {\tt ADD} costs $\FqM{41}+\FqS{7}+\Fqa{22}$, {\tt DBL} costs $\FqM{26}+\FqS{8}+2{\bf m}_c+\Fqa{25}$, {\tt mADD} costs $\FqM{32}+\FqS{5}+\Fqa{22}$ and {\tt mDBLADD} costs $\FqM{57}+\FqS{8}+\Fqa{42}$. Here ${\tt mADD}$ and ${\tt mDBLADD}$ compute $P \oplus Q$ and $[2]P\oplus Q$ where $P$ is projective and $Q$ (which is typically a fixed lookup table element) is normalized. We assume a fixed, signed window approach to computing one-dimensional scalar multiplications in $\Jac{\C}$; our experiments with 252-bit scalars found the window width $w=5$ to be optimal. In this case the scalar multiplications need $8 \times {\tt DBL} + 7 \times {\tt mADD}$ operations to build the lookup table of 16 elements, $\FqI{1}+\FqM{39}$ to normalize its entires, $200 \times {\tt DBL}+50 \times {\tt mDBLADD}$ operations in the main loop, and $\FqI{1}+\FqM{4}$ to normalize the output. For the two-dimensional scalar multiplications, we assume the standard approach to such multiexponentiations (cf.~\cite[\S 3]{GLV}). Our experiments showed that $w=2$ is optimal here, meaning the lookup table contains 16 elements $[u]P+[v]Q$ for $0 \leq u,v \leq 3$. In this case two-dimensional scalar multiplications need $3 \times {\tt DBL} + 10 \times {\tt mADD}$ operations to build the lookup table, $\FqI{1}+\FqM{36}$ to normalize its entries, $125 \times {\tt DBL} + 125 \times {\tt mDBLADD}$ operations in the main loop, and $\FqI{1}+\FqM{4}$ to normalize the output.

We only present counts for the dominant operations, i.e., the (multi)scalar multiplications, below. For the exact costs of compression and decompression, see Appendix~\ref{app:compression}; we omit details of point validation in \(\Jac{\C}\) and \(\Kfast{\C}\), both of which are essentially for free. For simplicity, we compare key generation, signing, and verification operations in~\S\ref{sub:schnorr-recover} assuming that no precomputation is performed offline, noting that, if space permits it, all of these operations can take advantage of precomputation in practice. 

\subsection{Schnorr signatures on the Gaudry-Schost Jacobian}
\label{sub:schnorr-recover}

Let the public generator, $P$, be any point of order $N$ in $\Jac{\C}(\FF_q)$.
Including $P$, 
generic elements $\langle x^2+a_1 x+a_0, b_1x+b_0\rangle$ in $\Jac{\C}(\FF_q)$ 
are compressed as in~\S\ref{app:compression} 
and encoded as a 256-bit string $(  {\tt bit}_0|| a_0 ||  {\tt bit}_1 ||a_1 )$. 

\subsubsection{Key generation:} 
Given the public generator $P$ 
and a 256-bit secret key $d$, 
compute $\mathcal{H}(d)=(d' || d'')$, 
where $d'$ and $d''$ are both $256$-bit strings. 
The public key is computed as the 256-bit encoding of $Q=[16 d']P$. 
Computing $(d',P) \mapsto Q$ directly on the Jacobian costs $8629{\bf M}+2131{\bf S}+424{\bf m}_c+7554{\bf a}+\FqI{2}$, but using Algorithm~\ref{algorithm:Ladder-template} to project, 
pseudomultiply on $\Kfast{\C}$, and recover, 
costs only $\FqM{2654}+\FqS{2312}+\Fqc{1546}+\Fqa{8221}+\FqI{2}$ -- see Theorem~\ref{thm:gaudrycost}. 

\subsubsection{Signing:} 
Given the public generator $P$, 
a message $M$,
and the secret key $\mathcal{H}(d)=(d' || d'')$, 
compute $r = \mathcal{H}(d''||M)$, 
then the 256-bit encoding of $R=[r]P$, 
then $h=\mathcal{H}(R||Q||M)$, 
and finally the 256-bit encoding of $s = (r-16hd') \bmod{N}$. 
The signature is the 512-bit string $(R||s)$. 
The costs of performing $(r,P) \mapsto [r]P$ 
are almost the same as above, 
except for the cost of the four final doublings in both scenarios: 
on $\Jac{\C}$ this comes to $8525{\bf M}+2099{\bf S}+416{\bf m}_c+7454{\bf a}+\FqI{2}$, 
while using Algorithm~\ref{algorithm:Ladder-template} 
costs $2654{\bf M}+2280{\bf S}+1522{\bf m}_c+8157{\bf a}+\FqI{2}$.

\subsubsection{Verification:} 
Given the public generator $P$, 
a message $M$, 
and a putative signature $(R||s)$ on \(M\) associated with a public key $Q$, 
compute $h=\mathcal{H}(R||Q||M)$ 
and accept the signature if $[16s]P+[16 h]Q=[16]R$;
otherwise, reject. 
Computing $((s,h),(P,Q)) \mapsto [16]([s]P+[h]Q)$ 
directly on the Jacobian with the costs $10813{\bf M}+2074{\bf S}+256{\bf m}_c+8670{\bf a}+\FqI{2}$. 
Using Algorithm~\ref{algorithm:DJB-template} 
to perform the 2-dimensional scalar multiplication via $\Kfast{\C}$ 
costs $4478{\bf M}+3325{\bf S}+2308{\bf m}_c+14272{\bf a}+\FqI{2}$ -- see Theorem~\ref{thm:gaudrycost}. \\

We summarize the costs of key generation, signing and verification in Table~\ref{table:sig-operations}. In addition to the speed benefits of Algorithm~\ref{algorithm:Ladder-template} and Algorithm~\ref{algorithm:DJB-template} in this scenario, we reiterate that their working on $\Kfast{\C}$ avoids the side-channel issues discussed in~\S\ref{sec:g2-side-channel}. 

\begin{table}
    \caption{
        Costs of key generation, signing and verification in the above Schnorr signature scheme. 
    }
    \label{table:sig-operations}
    \begin{center}
        \begin{tabular}{c|c|c|c|c|c|c}
            Algorithm    & Method &
           \ {\bf M} \ & \ {\bf S} \ & \ \({\bf m}_c\) \ & \, {\bf a} \, & \, {\bf I} \, \\
		\hline 
		\hline
		\multirow{ 2}{*}{ \ Key Generation \ }		     &	 \ fixed window mult. on \(\Jac{\C}\)	\	&
			8629	&	2131	&		424			&	7554		&	2		\\	
				      &	Algorithm~\ref{algorithm:Ladder-template} 	& 
			2654		&	2312	&		1546			&	8221		&	2		\\	
		\hline
		\hline
		\multirow{ 2}{*}{Signing}		     & fixed window mult. on \(\Jac{\C}\)		&
			8525		&	2099	&		416			&	7454		&	2		\\
				     &	Algorithm~\ref{algorithm:Ladder-template} 	&
			2654		&	2280	&		1522			&	8157		&	2		\\	
		\hline
		\hline
		\multirow{ 2}{*}{Verification}		     & multiscalar mult. on \(\Jac{\C}\)		&
			\ 10813 \		&	\ 2074 \ &		256			&	8670		&	2		\\
				     &	Algorithm~\ref{algorithm:DJB-template}	&
			4478		&	3325	&		\ 2308	\		&\	14272	\	&	\ 2	 \	\\	
            \hline
        \end{tabular}
    \end{center}
\end{table}

\bibliography{bib}
\bibliographystyle{plain}

\appendix

\section{
    Point compression in genus~2
}
\label{app:compression}

Here we show how to compress points on \(\Kgen{\C}\), \(\Kfast{\C}\) and \(\Jac{\C}\). We use ${\bf E}$ to denote a fixed exponentiation in the underlying finite field, e.g., to compute a square root. It is commonly the case that ${\bf E} \approx {\bf I}$. 


\subsubsection{Compressing points on the general Kummer \(\Kgen{\C}\).}
Generic
points $P=(\xi_1 \colon \xi_2 \colon \xi_3 \colon \xi_4)$ in
$\Kgen{\C}(\FF_q) \subset \PP^3(\FF_q)$ have $\xi_1 \neq 0$, 
so we can compute a normalized representative
\((1:k_2:k_3:k_4)\),
where each \(k_i = \xi_i/\xi_1\), at a cost of $\FqM{3}+\FqI{1}$.
We can further compress \(P\) to the data \((k_2,k_3,\mathtt{bit})\),
where \(\texttt{bit}\) is a single bit,
as follows:
the defining equation of \(\Kgen{\C}\) (Eq.~\eqref{eq:Kgen-equation}) 
is quadratic in $\xi_4$, so
\begin{align}\label{xi4}
    k_4
    =
    \frac{
        -K_1(1,k_2,k_3)
        \pm
        \sqrt{K_1(1,k_2,k_3)^2-4K_0(1,k_2,k_3)K_2(1,k_2,k_3)}
    }{
        2K_2(1,k_2,k_3)
    }
    \ ,
\end{align} 
which means that $k_4$ can be recovered (during decompression) 
from $k_2$, $k_3$, and the ``sign'' of the square root.
To compress, after computing \((k_2,k_3,k_4)\),
we compute $K_1=K_1(1,k_2,k_3)$ and $K_2=K_2(1,k_2,k_3)$ 
and set ${\tt bit}$ to be the sign bit\footnote{When $q$ is a large prime, the sign bit of an element is typically chosen as its parity when represented as an integer in $[0,q-1)$.} of $2K_2k_4+K_1$.
Computing the sign bit costs $\FqM{7}+\FqS{2}+\Fqa{11}$.

To decompress $(k_2,k_3,{\tt bit})$ to a point on \(\Kgen{\C}\), 
we first compute 
$K_0=K_0(1,k_2,k_3)$, 
$K_1=K_1(1,k_2,k_3)$
and 
$K_2=K_2(1,k_2,k_3)$. 
We can then use a simultaneous inversion-and-square-root\footnote{E.g., compute $\sqrt{u}$ and $1/v$ via $w \leftarrow uv^2$, $w \leftarrow w^{-1/2}$, then $(\sqrt{u},1/v) \leftarrow (uvw,uw)$.} to compute both $\sqrt{K_1^2-4K_0K_2}$ and $1/(2K_2)$ using one field
exponentiation, before
recovering $k_4$ via~\eqref{xi4}. 
The decompressed point is then \((1:k_2:k_3:k_4)\);
the decompression costs $\FqM{25}+\FqS{4}+\Fqa{24}+\FqE{1}$. 
   
\subsubsection{Compressing points on the fast Kummer \(\Kfast{\C}\).} 
As with the general Kummer \(\Kgen{\C}\),
generic points
$P=(X \colon Y \colon Z \colon T)$ 
in $\Kfast{\C}(\FF_q) \subset \PP^3(\FF_q)$ 
have a non-zero first coordinate \(X\),
so we might begin compressing \(P\)
by normalizing the \(X\)-coordinate to \(1\).
But the defining equation of \(\Kfast{\C}\)
(Eq.~\eqref{eq:Kfast-equation})
is quartic in all of its variables,
which suggests that compressing another coordinate would require solving an (unwieldy) during decompression. 
A faster approach to achieving further compression\footnote{This faster compression/decompression answers a question posed by Bernstein~\cite{dantalk}.} is to map $(X \colon Y \colon Z \colon T)$ to the general Kummer, and then to normalize before performing compression/decompression there. The cost of compressing $P=(X \colon Y \colon Z \colon T)$ into two $\FF_q$ elements and a single bit is therefore $\FqM{25}+\FqS{2}+\Fqa{23}+\FqI{1}$, and the cost of decompression $\FqM{40}+\FqS{4}+\Fqa{36}+\FqE{1}$.
 
%
%
%
   
\subsubsection{Compressing points on the Jacobian.} 
To compress the four Mumford coordinates, we follow Stahlke's
technique~\cite{stahlke}. For general genus~2 curves, compression of
$\langle x^2+a_1x+a_0, b_1x+b_0 \rangle$ into $(a_1,a_0, {\tt bit}_1, {\tt bit}_0)$ costs $3{\bf M}+1{\bf S}+4{\bf a}$, while decompression costs $36{\bf M}+5{\bf S}+45{\bf a}+2{\bf E}$. Compression costs the same for genus~2 curves in Rosenhain form, but decompression is significantly faster and requires $18 {\bf M}+4{\bf S}+27{\bf a}+2{\bf E}$. In the latter case, we compress by setting ${\tt bit}_1$ and ${\tt bit}_0$ as the least significant bits of $4(a_1b_1b_0-a_0b_1^2-b_0^2)$ and $b_1$ respectively. For decompression, we recover $b_1$ and $b_0$ from  $(a_1,a_0,{\tt bit}_1, {\tt bit}_0)$ by first computing
\begin{align}
A&=a_1^2-4a_0, \quad\quad  C= (a_0(a_0-f_3-a_1(a_1-f_4))+f_1)^2, \quad {\rm and} \nonumber\\
B&=2(f_4 a_0 (2 a_0-a_1^2)+a_0(f_3a_1-2f_2)+a_1(f_1+a_0(a_1^2-3a_0))),\nonumber
\end{align}
and using ${\tt bit}_1$ to choose $z_0$ as the correct root $z_0 = \sqrt{B^2-4AC}$. We then set $z_1 = (z_0-B)/(2A)$; here the square root and the inversion of $2A$ can again be combined into one exponentiation. We can then recover $b_1$ as $b_1=\sqrt{f_4(a_1^2-a_0)-a_1(f_3+a_1^2-2a_0)+f_2+z_1}$, using ${\tt bit}_0$ to choose the root. Finally, we recover $b_0$ as $b_0=(a_0(f_4a_1-f_3-q^2+a_0)+qz_1+f_1)/(2b_1)$, noting again that the square root and the inversion can be computed simultaneously. 

\end{document}